\documentclass[12pt]{amsart}
\usepackage{amsmath}
\usepackage{amssymb}
\usepackage[all]{xy}

\let\oldeqn\equation
\let\endoldeqn\endequation
\renewenvironment{equation}{\oldeqn\setlength{\thickmuskip}{10mu plus 5mu}}{\endoldeqn\relax}
\catcode`\*11
\let\oldeqn*\equation*
\let\endoldeqn*\endequation*
\renewenvironment{equation*}{\oldeqn*\setlength{\thickmuskip}{10mu plus 5mu}}{\endoldeqn*\relax}
\catcode`\*12

\newcommand{\Z}{\mathbb{Z}}

\newcommand{\C}{\mathbb{C}}

\newcommand{\fg}{\mathfrak{g}}
\newcommand{\tfg}{\widetilde{\mathfrak{g}}}
\newcommand{\fn}{\mathfrak{n}}
\newcommand{\fb}{\mathfrak{b}}

\newcommand{\fsl}{\mathfrak{sl}}
\newcommand{\fso}{\mathfrak{so}}
\newcommand{\fsp}{\mathfrak{sp}}

\newcommand{\Pp}{\mathbb{P}}
\newcommand{\cB}{\mathcal{B}}
\newcommand{\cN}{{\mathcal{N}}}
\newcommand{\tcN}{\widetilde{\mathcal{N}}}
\newcommand{\cO}{{\mathcal{O}}}
\newcommand{\cS}{\mathcal{S}}
\newcommand{\cF}{\mathcal{F}}
\newcommand{\tcF}{\widetilde{\mathcal{F}}}

\newcommand{\A}{\mathrm{A}}
\newcommand{\D}{\mathrm{D}}
\newcommand{\E}{\mathrm{E}}
\newcommand{\bv}{\mathbf{v}}
\newcommand{\bw}{\mathbf{w}}

\newcommand{\reg}{\mathrm{reg}}
\newcommand{\subreg}{\mathrm{subreg}}
\newcommand{\pr}{\mathrm{pr}}
\newcommand{\simto}{\mathrel{\overset{\sim}{\to}}}

\DeclareMathOperator{\Hom}{Hom}

\DeclareMathOperator{\Mat}{Mat}
\DeclareMathOperator{\Sing}{Sing}
\DeclareMathOperator{\Rad}{Rad}
\DeclareMathOperator{\tr}{tr}

\DeclareMathOperator{\rk}{rank}
\DeclareMathOperator{\im}{im}

\catcode`\@=11

\newlength{\tabwidth}
\newlength{\tabheight}
\newcommand{\tabstyle}{\textstyle}
\newlength{\tabrulel}
\newlength{\tabruler}
\newlength{\tabrulet}
\newlength{\tabruleb}
\newlength{\tabwidthx}
\newlength{\tabheightx}

\def\gentabbox#1#2#3#4#5#6#7{\vbox to \tabheight{%
 \setlength{\tabrulel}{#3}\setlength{\tabruler}{#4}%
 \setlength{\tabrulet}{#5}\setlength{\tabruleb}{#6}%
 \setlength{\tabwidthx}{#1\tabwidth}\addtolength{\tabwidthx}{0.5\tabrulel}%
   \addtolength{\tabwidthx}{0.5\tabruler}%
 \setlength{\tabheightx}{#2\tabheight}\addtolength{\tabheightx}{-\tabheight}%
 \hbox to #1\tabwidth{%
   \hspace{-0.5\tabrulel}\rule{\tabrulel}{#2\tabheight}\hspace{-\tabrulel}%
   \vbox to #2\tabheight{\hsize=\tabwidthx%
     \vspace{-0.5\tabrulet}\hrule width\tabwidthx height\tabrulet%
     \vspace{-0.5\tabrulet}\vss%
     \hbox to \tabwidthx{\hss#7\hss}%
       \vss\vspace{-0.5\tabruleb}%
     \hrule width\tabwidthx height\tabruleb\vspace{-0.5\tabruleb}}%
   \hspace{-\tabruler}\rule{\tabruler}{#2\tabheight}\hspace{-0.5\tabruler}}%
 \vspace{-\tabheightx}}}
\def\genblankbox#1#2{\vbox to \tabheight{\vfil\hbox to #1\tabwidth{\hfil}}}
\def\tabbox#1#2#3{\gentabbox{#1}{#2}{0.4pt}{0.4pt}{0.4pt}{0.4pt}{#3}}
\def\boldtabbox#1#2#3{\gentabbox{#1}{#2}{1.2pt}{1.2pt}{1.2pt}{1.2pt}{#3}}

\catcode`\:=13 \catcode`\.=13 \catcode`\;=13 
\catcode`\>=13 \catcode`\^=13
\def:#1\\{\hbox{$\tabstyle #1$}}
\def.#1{\tabbox{1}{1}{$\tabstyle #1$}}
\def>#1{\tabbox{2}{1}{$\tabstyle #1$}}
\def^#1{\tabbox{1}{2}{$\tabstyle #1$}}
\def;{\genblankbox{1}{1}\relax}
\catcode`\:=12 \catcode`\.=12 \catcode`\;=12 
\catcode`\>=12 \catcode`\^=7

\newenvironment{tableau}{\bgroup\catcode`\:=13 \catcode`\.=13
 \catcode`\;=13 \catcode`\>=13 \catcode`\^=13 
 \def\b##1##2##3{\boldtabbox{##1}{##2}{\vbox{##3}}}%
 \def\n##1##2##3{\tabbox{##1}{##2}{\vbox{##3}}}%
 \def\c{\tabbox{1}{1}{{}}}
 \vcenter\bgroup\offinterlineskip}{\egroup\egroup}

\catcode`\@=11

\newcommand{\row}[1]{\hbox{$\tabstyle #1$}}
\newcommand{\ghost}{\tabbox{1}{1}{$\cdot$}}

\numberwithin{equation}{section}
\newtheorem{theorem}{Theorem}[section]

\newtheorem{proposition}[theorem]{Proposition}
\newtheorem{corollary}[theorem]{Corollary}
\theoremstyle{definition}

\theoremstyle{remark}

\newtheorem{example}[theorem]{Example}

\title[Singularities of nilpotent orbit closures]{Singularities of nilpotent orbit closures}

\author{Anthony Henderson}
\address{School of Mathematics and Statistics\\
  University of Sydney, NSW 2006\\
  Australia}
\email{anthony.henderson@sydney.edu.au}


\thanks{The author was supported by ARC Future Fellowship No.~FT110100504.}

\begin{document}

\begin{abstract}
This is an expository article on the singularities of nilpotent orbit closures in simple Lie algebras over the complex numbers. It is slanted towards aspects that are relevant for representation theory, including Maffei's theorem relating Slodowy slices to Nakajima quiver varieties in type A. There is one new observation: the results of Juteau and Mautner, combined with Maffei's theorem, give a geometric proof of a result on decomposition numbers of Schur algebras due to Fang, Henke and Koenig.
\end{abstract}

\maketitle

%%%%%%%%%%%%%%%%%%%%%%%%%%%%%%%%%%%%%%%%%%%%%%%%%%%%%%%%%%%%%%%%%%%%%%%%%%%%%%%%%%%%%%%%%%

\section*{Introduction}

In September 2013, I had the privilege of giving a series of three lectures in the Japanese--Australian Workshop on Real and Complex Singularities (JARCS V), held at the University of Sydney. Why would someone who calls himself a representation theorist be speaking at a conference on singularities? One aim of the lectures was to answer that question, by indicating the important role that the study of singular varieties has in geometric representation theory. Another aim was simply to entertain an audience of singularity theorists with some beautiful examples of singularities and their deformations and resolutions: namely, those obtained by considering the closures of nilpotent orbits in simple Lie algebras.

This expository article is based on the slides from those lectures. I have added a few more details and references, while still assuming some familiarity with complex algebraic geometry. Experts in geometric representation theory will be able to think of many more results that ought to have been included to make a comprehensive survey. Indeed, one could imagine a sequel to Collingwood and McGovern's textbook \emph{Nilpotent orbits in semisimple Lie algebras}~\cite{cm}, entitled \emph{Singularities of closures of nilpotent orbits in semisimple Lie algebras}, which would be at least as long as the original. Two highly relevant topics which time and space constraints made it impossible to treat properly are the structure theory and classification of simple Lie algebras (for which see~\cite{humphreys}) and the theory of symplectic singularities (see~\cite{fu,ginzburg,kaledin}).

The goal that shaped my choice of topics and examples was Maffei's 2005 theorem~\cite[Theorem 8]{maffei} relating the resolutions of Slodowy slices in nilpotent orbit closures of type A to Nakajima quiver varieties (stated as Theorem~\ref{thm:maffei} below). Although this theorem is purely geometric, its importance really stems from its representation-theoretic context. Part of the motivation for Nakajima in conjecturing the isomorphism that Maffei proved was that both classes of varieties had been used in geometric constructions of the same representations, namely finite-dimensional representations of the Lie algebras $\fsl_m$. I do not tell that story in this article. Instead, in Section~\ref{sec:maffei}, I explain an application of Maffei's theorem to modular representation theory: by results of Juteau and Mautner (Theorem~\ref{thm:juteau-mautner} below), the decomposition numbers of symmetric groups and Schur algebras depend only on the singularities of the corresponding Slodowy slices, and one can use this to give a geometric proof of a result of Fang--Henke--Koenig~\cite[Corollary 7.1]{fanghenkekoenig}.   

\tableofcontents

%%%%%%%%%%%%%%%%%%%%%%%%%%%%%%%%%%%%%%%%%%%%%%%%%%%%%%%%%%%%%%%%%%%%%%%%%%%%%%%%%%%%%%%%%%

%----------------------------
\section{Adjoint orbits and the adjoint quotient}
%----------------------------

Let $\Mat_n$ be the vector space of $n\times n$ matrices over $\C$, and let $GL_n\subset\Mat_n$ be the group of invertible $n\times n$ matrices. 

Define a \emph{Lie algebra} to be a vector subspace $\fg$ of $\Mat_n$ that is closed under the commutator bracket $[X,Y]:=XY-YX$. In other words, $\fg$ must satisfy
\begin{equation*} 
X,Y\in\fg\Longrightarrow [X,Y]\in\fg.
\end{equation*}
To each such Lie algebra $\fg$ there corresponds a \emph{matrix group} $G$, which is defined to be the subgroup of $GL_n$ generated by all $\exp(X)$ for $X\in\fg$. Here $\exp$ denotes the matrix exponential, defined by the usual absolutely convergent series
\begin{equation*}
\exp(X)=I+X+\frac{1}{2}X^2+\frac{1}{6}X^3+\cdots,
\end{equation*}
where $I$ denotes the $n\times n$ identity matrix. The matrix $\exp(X)$ is always invertible because $\exp(X)\exp(-X)=I$.

One can easily show the following identity for any $X,Y\in\Mat_n$:
\begin{equation*}
\exp(X)Y\!\exp(-X)=Y+[X,Y]+\frac{1}{2}[X,[X,Y]]+\frac{1}{6}[X,[X,[X,Y]]]+\cdots.
\end{equation*} 
It follows that
\begin{equation*} 
g\in G,\ Y\in \fg \Longrightarrow gYg^{-1}\in\fg.
\end{equation*}
That is, the group $G$ acts on $\fg$ by conjugation: we write $g\cdot Y:=gYg^{-1}$ for short.
This is called the \emph{adjoint action} of $G$ on $\fg$, and its orbits are the \emph{adjoint orbits}.

We assume henceforth that $\fg$ is a \emph{simple} Lie algebra. This means that $\dim\fg>1$ and $\fg$ has no nontrivial subspace $S$ such that
\begin{equation*} 
X\in\fg,\ Y\in S\Longrightarrow [X,Y]\in S.
\end{equation*}
A consequence of this assumption (via a theorem of Cartan) is that the trace form $(X,Y)\mapsto \tr(XY)$ is nondegenerate on $\fg$. We can use this form to identify the space $\fg^*$ of linear functions on $\fg$ with $\fg$ itself, in a $G$-equivariant way. The algebra $\C[\fg]$ of polynomial functions on $\fg$ has a unique Poisson bracket extending the commutator bracket on $\fg$; thus $\fg$ is a Poisson variety and the adjoint action preserves this structure.

\begin{example}
The primary example of a simple Lie algebra is
\[ 
\begin{split}
\fsl_n&:=\{X\in\Mat_n\,|\,\tr(X)=0\},\ \text{ with corresponding group}\\
SL_n&:=\{g\in GL_n\,|\,\det(g)=1\},
\end{split}
\]
where $n\geq 2$.
The adjoint orbits in $\fsl_n$ are exactly the familiar \emph{similarity classes}. (Ordinarily one would define two matrices $X,Y$ to be similar if there is some $g\in GL_n$ such that $Y=gXg^{-1}$. But since any $g\in GL_n$ can be written as a product $g'g''$ where $g'\in SL_n$ and $g''$ is a scalar matrix, it comes to the same thing if one requires $g\in SL_n$.) 
\end{example}

We want to take a quotient of $\fg$ by the adjoint action of $G$. In the setting of algebraic geometry, this means that we need to consider the algebra $\C[\fg]^G$ of $G$-invariant polynomial functions on $\fg$. By definition, a function on $\fg$ is \emph{$G$-invariant} if and only if it is constant on each adjoint orbit. The classic result about these functions is:

\begin{theorem}[Chevalley]\label{thm:chevalley}
The algebra $\C[\fg]^G$ is freely generated by some homogeneous polynomials $\chi_1(X),\chi_2(X),\cdots,\chi_\ell(X)$. So as an algebra it is isomorphic to $\C[t_1,t_2,\cdots,t_\ell]$.
\end{theorem}
\noindent
The number of generators of $\C[\fg]^G$, denoted $\ell$, is called the \emph{rank} of $\fg$. 

\begin{example}
Let $\fg=\fsl_n$. Since the characteristic polynomial of a matrix $X$ is a similarity invariant, so is each of the coefficients of the characteristic polynomial. For matrices in $\fsl_n$, one of these coefficients, namely the trace, is zero by definition. The other $n-1$ coefficients provide $SL_n$-invariant polynomial functions: for $1\leq i\leq n-1$, we set
\[ 
\chi_i(X):=(-1)^{i+1}\times\text{coefficient of }t^{n-i-1}\text{ in }\det(tI-X).
\]
The sign here is chosen so that $\chi_i(X)$ equals the $(i+1)$th elementary symmetric polynomial in the eigenvalues of $X$. Note that $\chi_i(X)$ is a homogeneous polynomial function of $X$ of degree $i+1$.
It turns out that the functions $\chi_i(X)$ freely generate $\C[\fsl_n]^{SL_n}$, as per Theorem~\ref{thm:chevalley}. So $\fsl_n$ has rank $n-1$.
\end{example}

This is a convenient point at which to mention, for later reference, that there is a complete classification of simple Lie algebras (see~\cite{humphreys} for details). The isomorphism classes are labelled by symbols of the form $\mathrm{X}_\ell$ where $\mathrm{X}$ is one of the `Lie types' in the list $\A,\mathrm{B},\mathrm{C},\D,\E,\mathrm{F},\mathrm{G}$ and $\ell$ denotes the rank, which can take the following possible values:
\begin{itemize}
\item[$\A_\ell$:] $\ell$ can be any positive integer. The simple Lie algebra of type $\A_\ell$ is $\fsl_{\ell+1}$ (more correctly, $\fsl_{\ell+1}$ is a representative of the isomorphism class of simple Lie algebras of type $\A_\ell$).
\item[$\mathrm{B}_\ell$:] $\ell\geq 2$. The simple Lie algebra of type $\mathrm{B}_\ell$ is the special orthogonal Lie algebra $\fso_{2\ell+1}$, consisting of the skew-symmetric matrices in $\Mat_{2\ell+1}$.
\item[$\mathrm{C}_\ell$:] $\ell\geq 3$. The simple Lie algebra of type $\mathrm{C}_\ell$ is the symplectic Lie algebra $\fsp_{2\ell}$, whose definition we will not need.
\item[$\D_\ell$:] $\ell\geq 4$. The simple Lie algebra of type $\D_\ell$ is $\fso_{2\ell}$.
\item[$\E_\ell$:] $\ell\in\{6,7,8\}$.
\item[$\mathrm{F}_\ell$:] $\ell=4$ only.
\item[$\mathrm{G}_\ell$:] $\ell=2$ only. 
\end{itemize}
We will revisit this classification from a different viewpoint in Section~\ref{sec:min-degen}.

The \emph{adjoint quotient} map is the map obtained by combining all the generating $G$-invariant polynomial functions on $\fg$:
\[ \chi:\fg\to\C^{\ell}:X\mapsto(\chi_1(X),\cdots,\chi_\ell(X)). \]
One can then consider the fibres of this map, namely the affine algebraic varieties
\[
\chi^{-1}(u)=\{X\in\fg\,|\,\chi_1(X)=u_1,\;\cdots,\;\chi_\ell(X)=u_\ell\} 
\]
for $u=(u_1,\cdots,u_\ell)\in\C^\ell$. By definition, each such fibre is stable under the adjoint action, and hence is a union of some adjoint orbits; possibly more than one, because different orbits cannot necessarily be separated using $G$-invariant polynomial functions. 

These varieties have very good geometric properties.

\begin{theorem}[Kostant~\cite{kostant}] \label{thm:kostant}
For any $u\in\C^\ell$, the fibre $\chi^{-1}(u)$:
\begin{itemize}
\item consists of finitely many adjoint orbits -- for generic $u$, only one;
\item contains a unique dense orbit, and is hence irreducible;
\item has codimension $\ell$ in $\fg$, and is hence a complete intersection;
\item is nonsingular in codimension $1$, and is hence normal. 
\end{itemize}
\end{theorem}

\begin{example}
When $\fg=\fsl_n$, $\chi^{-1}(u)$ consists of matrices with a specified characteristic polynomial, or equivalently a specified multiset of eigenvalues. The different orbits in $\chi^{-1}(u)$ come from the different possible sizes of Jordan blocks; the dense orbit consists of the \emph{regular} matrices with the specified characteristic polynomial, i.e.\ those with a single Jordan block for each eigenvalue.
\end{example}

The nonsingularity in codimension $1$ of Theorem~\ref{thm:kostant} follows from the result (obtained independently by Borel, Kirillov and Kostant, see~\cite[Proposition 15]{kostant}) that every adjoint orbit has even dimension. This in turn follows from an easy calculation showing that the closure of any adjoint orbit is a Poisson subvariety of $\fg$, and the induced skew-symmetric form on the cotangent bundle of the orbit is nondegenerate.

%-------------------------
\section{Nilpotent orbits and their closures}
%-------------------------

We now come to the main object of study. The \emph{nilpotent cone} $\cN\subset\fg$ is one of the fibres of the adjoint quotient, namely $\cN:=\chi^{-1}(0)$. Since it is defined by the vanishing of some homogeneous polynomials, $\cN$ is stable under scalar multiplication (this is the sense in which it is a cone). By Theorem~\ref{thm:kostant}, $\cN$ is the union of finitely many adjoint orbits, which are called the \emph{nilpotent orbits} of $\fg$. It follows immediately from this finiteness that each nilpotent orbit is stable under nonzero scalar multiplication.

If $\cO$ is a nilpotent orbit, then its closure $\overline{\cO}$ is stable under the adjoint action and contained in $\cN$, so it is the union of $\cO$ and some other nilpotent orbits $\cO'$ which must have smaller dimension than $\cO$. This gives rise to a partial order $\leq$ on the nilpotent orbits, the \emph{closure order}, where $\cO'\leq\cO$ means that $\cO'\subseteq\overline{\cO}$. There is always a unique minimum orbit for this partial order, namely the orbit $\{0\}$ consisting solely of $0$; by Theorem~\ref{thm:kostant}, there is also always a unique maximum orbit for this partial order, namely the \emph{regular} orbit $\cO_{\reg}$ which is dense in $\cN$.

\begin{example} \label{ex:sln-orbits}
When $\fg=\fsl_n$, $\cN$ consists of all $n\times n$ matrices that are \emph{nilpotent} in the sense that their characteristic polynomial is the same as that of the zero matrix; or equivalently, that their only eigenvalue is $0$; or equivalently, that some power of them equals the zero matrix.
In this case the splitting of $\cN$ into nilpotent orbits is determined purely by the different sizes of Jordan blocks a nilpotent matrix can have. The order of the blocks being immaterial, it is conventional to list the sizes in decreasing order. Thus the set of nilpotent orbits is in bijection with the set of \emph{partitions} of $n$, where a partition $\lambda$ of $n$ means a weakly decreasing sequence $(\lambda_1,\lambda_2,\cdots)$ of nonnegative integers adding up to $n$. (It is notationally convenient to make the sequence infinite, with all terms being zero after a certain point; but in writing specific partitions we omit the zeroes.) We write $\cO_\lambda$ for the nilpotent orbit corresponding to the partition $\lambda$. In this notation, $\{0\}=\cO_{(1,1,\cdots,1)}$ and $\cO_{\reg}=\cO_{(n)}$.
\end{example}

\begin{example}
As a sub-case of the previous example, take $\fg=\fsl_2$. Then the nilpotent cone is a singular quadric hypersurface:
\[ \cN=\left\{\left.\begin{bmatrix}a&b\\c&-a\end{bmatrix}\,\right|\,a,b,c\in\C,\,a^2+bc=0\right\}.\]
In this case, $\{0\}$ and $\cO_{\reg}$ are the only two orbits (i.e.\ $\cO_{\reg}=\cN\setminus\{0\}$).
\end{example}

\begin{example} \label{ex:sl3-orbits}
As another sub-case of Example~\ref{ex:sln-orbits}, take $\fg=\fsl_3$. There are three partitions of $3$: $(1,1,1)$, $(2,1)$ and $(3)$. Hence we have three nilpotent orbits in $\fsl_3$, and the closure order is such that
\[
\{0\}<\cO_{(2,1)}<\cO_{\reg}.
\] 
Note that
\begin{equation*}
\begin{split}
\cO_{(2,1)}&=\{X\in\fsl_3\,|\,\rk(X)=1\},\text{ so}\\
\overline{\cO_{(2,1)}}&=\{X\in\fsl_3\,|\,\rk(X)\leq 1\}\\
&=\{X\in\fsl_3\,|\,\text{every $2\times 2$ minor of }X\text{ is }0\}.
\end{split}
\end{equation*}
Here we have used the fact that a rank-$1$ matrix with trace $0$ is automatically nilpotent. Thus the closed subvariety $\overline{\cO_{(2,1)}}$ of $\fsl_3$ is defined by the vanishing of nine degree-$2$ polynomials, and none of these defining equations is redundant. Since the codimension of $\overline{\cO_{(2,1)}}$ in $\fsl_3$ is easily seen to be $4$, we conclude that $\overline{\cO_{(2,1)}}$ does not share with $\cN$ the property of being a complete intersection. (In fact, $\overline{\cO}$ is never a complete intersection unless $\cO=\{0\}$ or $\cO=\cO_\reg$; see~\cite{brionfu,namikawa}.)
\end{example}

We write $\cO'\prec\cO$ to mean that $\cO$ covers $\cO'$ in the closure order, in the sense that $\cO'<\cO$ and there is no orbit $\cO''$ such that $\cO'<\cO''<\cO$. In this case one says that $\cO'$ is a \emph{minimal degeneration} of $\cO$. 

It is a fact that there is always a unique orbit $\cO_{\min}$ such that $\{0\}\prec\cO_{\min}$, the \emph{minimal orbit} (short for `minimal nonzero orbit'). There is also always a unique orbit $\cO_{\subreg}$ such that $\cO_{\subreg}\prec\cO_{\reg}$, the \emph{subregular orbit}. See~\cite[Chapter 4]{cm} for the proofs.

In Example~\ref{ex:sln-orbits} we saw a combinatorial parametrization of the nilpotent orbits for $\fsl_n$. We can express the closure order explicitly using these parameters. 
\begin{theorem}[Gerstenhaber~\cite{gerstenhaber}] \label{thm:gerstenhaber}
The closure order on nilpotent orbits for $\fsl_n$ is given by:
\[
\cO_\mu\leq\cO_\lambda\quad\Longleftrightarrow\quad
\begin{array}{rcl}
\mu_1&\leq&\lambda_1,\\
\mu_1+\mu_2&\leq&\lambda_1+\lambda_2,\\
\mu_1+\mu_2+\mu_3&\leq&\lambda_1+\lambda_2+\lambda_3,\\
\vdots\qquad\vdots\quad&\vdots&\quad\vdots\qquad\vdots
\end{array}
\]
\end{theorem}
\noindent
The partial order on partitions arising in Theorem~\ref{thm:gerstenhaber} is called the \emph{dominance order}, written $\mu\trianglelefteq\lambda$.

To describe the minimal degenerations in this case, it is helpful to use the graphical representation of a partition $\lambda=(\lambda_1,\lambda_2,\cdots)$ as a left-justified diagram of boxes, with $\lambda_i$ boxes in the $i$th row from the top. A \emph{corner box} of such a diagram means a box which has no box either below or to the right of it.
\begin{proposition}[{\cite[Section 1]{kp1}}] \label{prop:sln-mindeg} 
The minimal degenerations of nilpotent orbits for $\fsl_n$ are described as follows:
\[
\cO_\mu\prec\cO_\lambda\quad\Longleftrightarrow\quad
\begin{array}{l}
\text{$\mu$ differs from $\lambda$ by moving a corner box}\\
\textbf{either}\ \text{down from one row to the next,}\\
\textbf{or}\  \text{left from one column to the next.}
\end{array}
\]
When $\cO_\mu\prec\cO_\lambda$, the codimension of $\cO_\mu$ in $\overline{\cO_\lambda}$ is twice the difference in row numbers of the box that is moved.
\end{proposition}

\begin{example}
Take $\fg=\fsl_{16}$ and let $\lambda=(5,4,4,3)$. There are three corner boxes, of which two can be moved down one row or left one column; these corner boxes are indicated in bold in the following pictures.
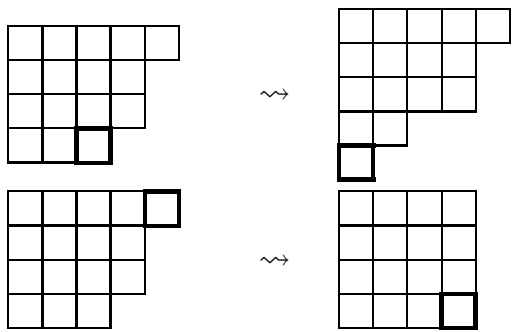
\[
\begin{split}
\begin{tableau}
\row{\c\c\c\c\c}
\row{\c\c\c\c}
\row{\c\c\c\c}
\row{\c\c\b{1}{1}{{}}}
\end{tableau}
\qquad&\rightsquigarrow\quad
\begin{tableau}
\row{\c\c\c\c\c}
\row{\c\c\c\c}
\row{\c\c\c\c}
\row{\c\c}
\row{\b{1}{1}{{}}}
\end{tableau}
\\
\begin{tableau}
\row{\c\c\c\c\b{1}{1}{{}}}
\row{\c\c\c\c}
\row{\c\c\c\c}
\row{\c\c\c}
\end{tableau}
\qquad&\rightsquigarrow\quad
\begin{tableau}
\row{\c\c\c\c}
\row{\c\c\c\c}
\row{\c\c\c\c}
\row{\c\c\c\b{1}{1}{{}}}
\end{tableau}
\end{split}
\]
So the minimal degenerations of $\cO_{(5,4,4,3)}$ are $\cO_{(5,4,4,2,1)}$ (of codimension $2$) and $\cO_{(4,4,4,4)}$ (of codimension $6$).
\end{example}

\begin{example} \label{ex:sln-min-subreg}
When $\fg=\fsl_n$, the minimal orbit $\cO_{\min}$ is $\cO_{(2,1,1,\cdots,1)}$, and its dimension (which equals the codimension of $\{0\}=\cO_{(1,1,\cdots,1)}$ in its closure) equals $2(n-1)$. The subregular orbit $\cO_{\subreg}$ is $\cO_{(n-1,1)}$, and its codimension in $\cN=\overline{\cO_{(n)}}$ is $2$.
\end{example}

Each nilpotent orbit $\cO$ is a nonsingular variety, since it is a homogeneous space for the group $G$. However, the closure $\overline{\cO}$ has interesting singularities. As was observed by Namikawa, $\overline{\cO}$ is a \emph{holonomic} Poisson variety in the sense of~\cite[Definition 1.3]{kaledin}, and it follows as in~\cite[Lemma 1.4]{kaledin} that the singular locus $\Sing(\overline{\cO})$ is the whole of $\overline{\cO}\setminus\cO$. So the irreducible components of $\Sing(\overline{\cO})$ are the closures $\overline{\cO'}$ where $\cO'\prec\cO$; in other words, the \emph{generic singularities} of $\overline{\cO}$ are those at points of the minimal degenerations $\cO'\prec\cO$.

\section{Slodowy slices}

\label{sec:slodowy}

In studying the singularity of $\overline{\cO}$ at a point $X$, the first step is to discard the directions in which $\overline{\cO}$ is nonsingular, that is, the directions of the orbit $\cO_X=G\cdot X$. There is a particularly nice way to do this that relies on the following result from Lie algebra structure theory:
\begin{theorem}[Jacobson--Morozov, see~{\cite[Section 3.3]{cm}}] \label{thm:jm}
Given any $X\in\cN$, there is an element $Y\in\fg$ such that
\[ 
[[X,Y],X]=2X,\ [[X,Y],Y]=-2Y.
\]
Any such $Y$ belongs to $\cN$, and in fact belongs to the $G$-orbit $\cO_X$. Moreover, if $Y,Y'$ are two such elements, then there is some $g\in G$ such that $g\cdot X=X$ and $g\cdot Y=Y'$.
\end{theorem}

\begin{example}
If $\fg=\fsl_2$ and $X=[\begin{smallmatrix}0&1\\0&0\end{smallmatrix}]$, then it is easy to check that $Y=[\begin{smallmatrix}0&0\\1&0\end{smallmatrix}]$ satisfies the conditions in Theorem~\ref{thm:jm}. In fact, this example is the motivation for those conditions: if $X,Y$ satisfy the conditions of Theorem~\ref{thm:jm} and are nonzero, then they generate a subalgebra of $\fg$ that is isomorphic to $\fsl_2$.
\end{example}

Given $X,Y\in\cN$ as in Theorem~\ref{thm:jm}, we define
\[
\cS_X := \{Z\in\fg\,|\,[Z-X,Y]=0\}.
\]
Note that $\cS_X$ is an affine-linear subspace of $\fg$ passing through $X$, called the \emph{Slodowy slice} at $X$. It would be more correct to denote it $\cS_{X,Y}$, but we omit $Y$ from the notation on the grounds that it is `almost' determined by $X$, in the sense of the last statement of Theorem~\ref{thm:jm}; if we were to change $Y$, then $\cS_X$ would only change by the action of some $g\in G$. Nevertheless, we must bear in mind that the group that acts naturally on $\cS_X$ is not the whole stabilizer $G_X$ of $X$, but rather the joint stabilizer $G_{X,Y}=\{g\in G\,|\,g\cdot X=X,g\cdot Y=Y\}$.

\begin{proposition}[Kostant, Slodowy~\cite{slod}] \label{prop:slodowy}
If $X,Y\in\cN$ are as in Theorem~\ref{thm:jm}, then $\cS_X$ is a transverse slice to $\cO_X$ at $X$, in the sense that $\dim\cS_X=\dim\fg-\dim\cO_X$ and the morphism
\[ G\times \cS_X\to\fg:(g,X')\mapsto g\cdot X' \] 
is a submersion.
\end{proposition} 
\noindent
This result means that for any orbit closure $\overline{\cO}$ containing $X$, the singularity of $\overline{\cO}$ at $X$ is \emph{smoothly equivalent} to that of $\cS_X\cap\overline{\cO}$ at $X$. Henceforth we will usually consider, instead of $\overline{\cO}$ itself, such a subvariety $\cS_X\cap\overline{\cO}$. Note that if $X=0$, then necessarily $Y=0$ also, so $\cS_0=\fg$ and $\cS_0\cap\overline{\cO}=\overline{\cO}$.

The dimension of $\cS_X\cap\overline{\cO}$ equals the codimension of $\cO_X$ in $\overline{\cO}$.
It is a fact that $\cS_X$ does not meet any nilpotent orbits except those whose closure contains $X$, so $\cS_X\cap\overline{\cO}$ only meets orbits $\cO'$ such that $\cO_X\leq \cO'\leq \cO$. Moreover, $\cS_X\cap\cO_X=\{X\}$. In particular, if $\cO_X\prec\cO$, then $\cS_X\cap\overline{\cO}$ is singular only at $X$; that is, we have an \emph{isolated singularity} in that case.

\begin{example}
Take $\fg=\fsl_3$. A pair $X,Y$ lying in the regular nilpotent orbit $\cO_{\reg}$ is
\[ X=\begin{bmatrix}0&1&0\\0&0&1\\0&0&0\end{bmatrix},\ Y=\begin{bmatrix}0&0&0\\2&0&0\\0&2&0\end{bmatrix}. \]
In this case a simple calculation gives
\[
\cS_X=\left\{\left.\begin{bmatrix}0&1&0\\a&0&1\\b&a&0\end{bmatrix}\right| a,b\in\C\right\}.
\]
Computing the characteristic polynomial of the given matrix in $\cS_X$, one finds 
\[
\chi_1\left(\begin{bmatrix}0&1&0\\a&0&1\\b&a&0\end{bmatrix}\right)=-2a,\qquad 
\chi_2\left(\begin{bmatrix}0&1&0\\a&0&1\\b&a&0\end{bmatrix}\right)=b.
\]
So $\cS_X$ meets each fibre of the adjoint quotient map in exactly one point (it turns out that this holds for general $\fg$ when $X\in\cO_{\reg}$). In particular, $\cS_X\cap\cN=\{X\}$, in accordance with the above general rules.
\end{example}

\begin{example} \label{ex:sl3-subreg}
Now keep $\fg=\fsl_3$ but consider a pair in the subregular nilpotent orbit $\cO_{(2,1)}$:
\[
X=\begin{bmatrix}0&1&0\\0&0&0\\0&0&0\end{bmatrix},\ Y=\begin{bmatrix}0&0&0\\1&0&0\\0&0&0\end{bmatrix}. \]
A simple calculation gives 
\[
\cS_X=\left\{\left.\begin{bmatrix}a&1&0\\b&a&c\\d&0&-2a\end{bmatrix}\right| a,b,c,d\in\C\right\}.
\]
Computing the characteristic polynomial, one finds:
\[
\begin{split}
\chi_1\left(\begin{bmatrix}a&1&0\\b&a&c\\d&0&-2a\end{bmatrix}\right)&=-3a^2-b,\\
\chi_2\left(\begin{bmatrix}a&1&0\\b&a&c\\d&0&-2a\end{bmatrix}\right)&=2a(b-a^2)+cd.
\end{split}
\]
In particular, 
\[ \cS_X\cap\cN=\left\{\left.\begin{bmatrix}a&1&0\\-3a^2&a&c\\d&0&-2a\end{bmatrix}\;\right|\; a,c,d\in\C,\; 8a^3=cd\right\}, \]
a singular surface with an isolated singularity at $X$.
\end{example}

%----------------------------
\section{Why do representation theorists care?}
%----------------------------

Perhaps surprisingly, the singularities of nilpotent orbit closures have been shown to encode a lot of representation-theoretic information. In particular, the case of $\fsl_n$ relates to the representations of the symmetric group $S_n$. To give an example of a concrete statement along these lines, we need to introduce some classic representation theory constructions. 

For any partition $\lambda$ of $n$, define the polynomial
\[
\begin{split}
\pi_\lambda(x_1,\cdots,x_n)&:=\Delta(x_1,\cdots,x_{\lambda_1})\Delta(x_{\lambda_1+1},\cdots,x_{\lambda_1+\lambda_2})\cdots,\text{ where}\\
\Delta(y_1,\cdots,y_m)&:=\prod_{1\leq i<j\leq m} y_i-y_j. 
\end{split}
\]
Define the \emph{Specht module} $S_\lambda\subset\Z[x_1,\cdots,x_n]$ by 
\[
S_\lambda=\text{$\Z$-span}\{\pi_\lambda(x_{\sigma(1)},\cdots,x_{\sigma(n)})\,|\,\sigma\in S_n\}.
\] 
Experts will note that this Specht module is traditionally labelled not by $\lambda$ but by the transpose partition $\lambda^{\mathbf{t}}$ (so the conventions in some subsequent results are tranposed from their familiar form).
However it is labelled, $S_\lambda$ is clearly stable under the action of the symmetric group $S_n$ (acting by permuting the variables $x_1,\cdots,x_n$); that is, it is a $\Z S_n$-module. Let $B_\lambda:S_\lambda\times S_\lambda\to\Z$ be the restriction of the $\Z$-bilinear form on $\Z[x_1,\cdots,x_n]$ for which the monomials are orthonormal. 

\begin{example}
Take $n=3$. By convention, $\Delta(y_1,\cdots,y_m)=1$ when $m<2$ (since it is then an empty product), so
\[
\begin{split}
\pi_{(1,1,1)}(x_1,x_2,x_3)&=1,\\
\pi_{(2,1)}(x_1,x_2,x_3)&=x_1-x_2,\\
\pi_{(3)}(x_1,x_2,x_3)&=(x_1-x_2)(x_1-x_3)(x_2-x_3).
\end{split}
\]
Thus, $S_{(1,1,1)}$ is a free rank-one $\Z$-module spanned by $1$, on which $S_3$ acts trivially, and $B_{(1,1,1)}(1,1)=1$. By contrast, 
\[ S_{(2,1)}=\text{$\Z$-span}\{x_1-x_2,x_1-x_3,x_2-x_3\} \] 
is a free rank-two $\Z$-module with basis $\{x_1-x_2,x_2-x_3\}$, and relative to this basis the matrix of the form $B_{(2,1)}$ is $[\begin{smallmatrix}2&-1\\-1&2\end{smallmatrix}]$. Finally, $S_{(3)}$ is a free rank-one $\Z$-module spanned by $\pi:=(x_1-x_2)(x_1-x_3)(x_2-x_3)$, on which $S_3$ acts via its sign character, and $B_{(3)}(\pi,\pi)=6$. 
\end{example}

\begin{theorem}[James~\cite{james}] \label{thm:james}
For any field $F$, let $S_\lambda^F=S_\lambda\otimes_\Z F$ and let $B_\lambda^F$ be the $F$-bilinear form on $S_\lambda^F$ induced by $B_\lambda$.
\begin{enumerate}
\item When $F$ has characteristic $0$, a complete set of inequivalent irreducible representations of $S_n$ over $F$ is given by
\[ \{S_\lambda^F\,|\,\lambda\text{ is a partition of }n\}. \]
\item When $F$ has characteristic $p$, a complete set of inequivalent irreducible representations of $S_n$ over $F$ is given by
\[ \{D_\mu^F\,|\,\mu\text{ is a $p$-restricted partition of }n\}, \]
where $D_\mu^F=S_\mu^F/\Rad(B_\mu^F)$, and $\mu$ is said to be \emph{$p$-restricted} if $\mu_i-\mu_{i+1}<p$ for all $i$.
\end{enumerate}
\end{theorem}

\begin{example}
Continue the $n=3$ example. If $F$ has characteristic $0$, then the representations $S_{(1,1,1)}^F$, $S_{(2,1)}^F$, $S_{(3)}^F$ are all irreducible and inequivalent; they are referred to respectively as the trivial, reflection, and sign representations of $S_3$ over $F$. If $F$ has characteristic $p>3$, the situation is the same: all the partitions $\mu$ of $3$ are $p$-restricted, and for all such $\mu$ we have $D_\mu^F=S_\mu^F$ (that is, the form $B_\mu^F$ is nondegenerate; in other words, the determinant of the matrix of the integral form $B_\mu$ is not divisible by $p$). If $F$ has characteristic $2$, then $S_{(1,1,1)}^F$, $S_{(2,1)}^F$, $S_{(3)}^F$ are again all irreducible, but $S_{(3)}^F$ is equivalent to $S_{(1,1,1)}^F$ (the sign representation is trivial since $-1=1$ in $F$). This accords with Theorem~\ref{thm:james}, since $D_\mu^F=S_\mu^F$ for $\mu\in\{(1,1,1),(2,1)\}$, and the partition $(3)$ is not $2$-restricted. If $F$ has characteristic $3$, then $S_{(2,1)}^F$ is reducible: the one-dimensional subspace spanned by $x_1+x_2+x_3$ is invariant, and equals the radical of $B_{(2,1)}^F$ (note that the matrix of the integral form $B_{(2,1)}$ has determinant $3$). So $D_{(2,1)}^F$ is the quotient of $S_{(2,1)}^F$ by this subspace, and is equivalent to the sign representation. Again $D_{(1,1,1)}^F=S_{(1,1,1)}^F$, and the partition $(3)$ is not $3$-restricted.        
\end{example}

The representations $S_{\lambda}^F$ are fairly well understood, in the sense that there are combinatorial formulas for their dimensions, explicit bases, and many other results. The irreducible representations $D_\mu^F$, defined when $F$ has characteristic $p$ and $\mu$ is $p$-restricted, have proved harder to handle.
A major unsolved problem in representation theory is to compute the \emph{decomposition numbers} $d_{\lambda\mu}^p:=[S_\lambda^F:D_\mu^F]$, which count the occurrences of each irreducible $D_\mu^F$ in a composition series for $S_\lambda^F$. If we knew these numbers, we would be able to translate our knowledge of $S_\lambda^F$ into knowledge of $D_\mu^F$. There is actually a more general definition of $d_{\lambda\mu}^p$ which makes sense when $\mu$ is not $p$-restricted, involving representations of the \emph{Schur algebra} rather than the symmetric group (see~\cite{mathas}). 

One of the earliest results proved about these decomposition numbers was that $d_{\lambda\mu}^p=0$ unless $\mu\trianglelefteq\lambda$, in the notation introduced after Theorem~\ref{thm:gerstenhaber}; in other words, $d_{\lambda\mu}^p=0$ unless $\cO_\mu\subseteq\overline{\cO_\lambda}$. In retrospect, this can be seen to be a hint that there should be a connection with the closures of nilpotent orbits for $\fsl_n$, and such a connection has now been established by Juteau~\cite{juteau,jmw} (in the cases relating to $S_n$) and Mautner~\cite{mautner} (in the setting of the Schur algebra). Their results imply:

\begin{theorem}[Juteau, Mautner] \label{thm:juteau-mautner}
Let $\lambda$ and $\mu$ be partitions of $n$ with $\mu\trianglelefteq\lambda$.
The decomposition number $d_{\lambda\mu}^p$ depends only on $p$ and the singularity of $\cS_X\cap\overline{\cO_\lambda}$ at $X$ for $X\in\cO_\mu$.
\end{theorem}

More precisely, what one needs to know about $\cS_X\cap\overline{\cO_\lambda}$ are certain local intersection cohomology groups with coefficients in the finite field with $p$ elements; these are invariants of the smooth equivalence class of the singularity at $X$. It may well be that calculating these local intersection cohomology groups in full is no easier than the algebraic problem of calculating the decomposition numbers $d_{\lambda\mu}^p$, but as we will see in Section~\ref{sec:maffei}, the geometric approach allows enlightening proofs of some qualitative results.   

%-----------------
\section{Kleinian singularities}
\label{sec:min-degen}
%-----------------

In what appears at first to be a digression from the study of nilpotent orbits, we recall the definition of a famous class of isolated surface singularities. Let $\Gamma$ be a nontrivial finite subgroup of $SL_2(\C)$. Up to conjugacy in $SL_2(\C)$, there is a quite restricted range of possibilities for such $\Gamma$ (see~\cite{mckay,slod} for more details):
\begin{itemize}
\item[$\A_\ell$:] cyclic of order $\ell+1$ for $\ell\geq 1$;
\item[$\D_\ell$:] binary dihedral of order $4(\ell-2)$ for $\ell\geq 4$;
\item[$\E_\ell$:] binary tetrahedral (order $24$), binary octahedral (order $48$), binary icosahedral (order $120$) for $\ell=6,7,8$ respectively.
\end{itemize}
The various types of groups $\Gamma$ are labelled here by some of the symbols $\mathrm{X}_\ell$ used in the classification of simple Lie algebras; the reason for this will be seen shortly.

Since $\Gamma$ is a subgroup of $SL_2(\C)$, it comes with an action on the vector space $\C^2$. We now want to consider the quotient $\C^2/\Gamma$, which by definition is the affine variety whose algebra of functions is the invariant ring $\C[\C^2]^\Gamma$. 
\begin{theorem}[Klein]
In each case the invariant ring $\C[\C^2]^\Gamma$ is generated by three homogeneous elements $x,y,z$ satisfying a single relation, given in the following table.
\[
\begin{array}{|l|l|}
\hline
\A_\ell&x^{\ell+1}+yz=0\\
\D_\ell&x^{\ell-1}+xy^2+z^2=0\\
\E_6&x^4+y^3+z^2=0\\
\E_7&x^3y+y^3+z^2=0\\
\E_8&x^5+y^3+z^2=0\\
\hline
\end{array}
\]
\end{theorem}
\noindent
So $\C^2/\Gamma$ can be identified with the hypersurface in $\C^3$ defined by the equation given in the above table; it is a normal surface with an isolated singularity at $0$, known as a \emph{Kleinian singularity}.

The \emph{semiuniversal deformation} of a hypersurface~$P(x_1,\cdots,x_d)=0$ in $\C^d$ with an isolated singularity at $0$ is a family of hypersurfaces in $\C^d$, depending on a parameter $u=(u_i)\in\C^\ell$, defined by the equations
\[
P(x_1,\cdots,x_d)+\sum_{i=1}^\ell u_i b_i(x_1,\cdots,x_d)=0,
\]
where $\{b_i\}$ is a linearly independent subset of $\C[x_1,\cdots,x_d]$ whose span is complementary to the ideal $(P,\frac{\partial P}{\partial x_1},\cdots,\frac{\partial P}{\partial x_d})$. So the dimension $\ell$ of the parameter space equals the codimension of this ideal, which is finite because the singularity is isolated.

\begin{example} \label{ex:typea-semiuniversal}
Consider the type-$\A_\ell$ Kleinian singularity, given by the polynomial $P(x,y,z)=x^{\ell+1}+yz$. The ideal $(P,\frac{\partial P}{\partial x},\frac{\partial P}{\partial y},\frac{\partial P}{\partial z})$ is $(x^\ell,y,z)$, so the two meanings of the letter $\ell$ are consistent, and the set $\{b_i\}$ can be chosen to be $\{x^{\ell-1},\cdots,x,1\}$. Thus, the equation defining a general hypersurface in the semiuniversal deformation is
\[
x^{\ell+1}+u_1 x^{\ell-1}+\cdots u_{\ell-1} x+ u_\ell + yz = 0.
\] 
\end{example}

Now recall that in our simple Lie algebra $\fg$ there is a unique \emph{subregular} nilpotent orbit $\cO_{\subreg}\prec\cO_{\reg}$; it turns out that the codimension of $\cO_{\subreg}$ in $\overline{\cO_\reg}=\cN$ is always $2$ (in the case $\fg=\fsl_n$, we saw this in Example~\ref{ex:sln-min-subreg}). Let $\cS_{\subreg}$ denote the Slodowy slice $\cS_X$ for some $X\in\cO_{\subreg}$. By the results of Section~\ref{sec:slodowy}, $\cS_{\subreg}\cap\cN$ is a normal surface with an isolated singularity. The explanation for the above labelling of subgroups $\Gamma$ of $SL_2(\C)$ is the following remarkable connection between the two classes of surface singularities:

\begin{theorem}[Brieskorn, see~\cite{slod}] \label{thm:brieskorn}
If $\fg$ is of type $\A_\ell$, $\D_\ell$ or $\E_\ell$ and $\Gamma\subset SL_2(\C)$ is of the corresponding type, we have an isomorphism
\[ \cS_{\subreg}\cap\cN \cong \C^2/\Gamma. \]
Moreover, the family of varieties $\cS_{\subreg}\cap\chi^{-1}(u)$, where $\chi$ is the adjoint quotient map and $u$ runs over $\C^\ell$, is isomorphic to the semiuniversal deformation of $\C^2/\Gamma$.
\end{theorem}

\begin{example}
Let $\fg=\fsl_3$, which is of type $\A_2$. Then $\Gamma$ is cyclic of order $3$, and the Kleinian singularity $\C^2/\Gamma$ is given by the equation $x^3+yz=0$. The calculation we did in Example~\ref{ex:sl3-subreg} already makes it clear that $\cS_{\subreg}\cap\cN$ is isomorphic to $\C^2/\Gamma$, and the remaining part of Theorem~\ref{thm:brieskorn} is equally easy to check.
\end{example}

\begin{example} \label{ex:sl4-subreg}
Let $\fg=\fsl_4$, which is of type $\A_{3}$. Calculations similar to those in Example~\ref{ex:sl3-subreg} give:
\[
\cS_{\subreg}=\left\{\begin{bmatrix}
a&1&0&0\\
b&a&1&0\\
c&b&a&d\\
e&0&0&-3a
\end{bmatrix}\right\},\ 
\begin{array}{l} 
\chi_1=-6a^2-2b,\\ 
\chi_2=-8a^3+4ab+c,\\
\chi_3=-3a^4+6a^2b-3ac-de.
\end{array}
\]
Setting $\chi_i=u_i$ and solving for $b$ and $c$, one sees that $\cS_{\subreg}\cap\chi^{-1}(u)$ is isomorphic to the following hypersurface in $\C^3$:
\[
\{(a,d,e)\in\C^3\,|\,81a^4+de+9u_1a^2+3u_2a+u_3=0\}.
\]
This clearly gives a family isomorphic to the semiuniversal deformation of the Kleinian singularity of type $\A_3$, as described in Example~\ref{ex:typea-semiuniversal}. For comparison with the next example, we record that here the group $G_{X,Y}$ is isomorphic to $\C^\times$, and its action on $\cS_{\subreg}$ (which necessarily preserves each intersection $\cS_{\subreg}\cap\chi^{-1}(u)$) is the action of $\C^\times$ given by fixing $a,b,c$ and scaling $d$ and $e$ by inverse scalars.
\end{example}

The obvious question left unanswered by Theorem~\ref{thm:brieskorn} is what happens for the simple Lie algebras of types other than $\A_\ell,\D_\ell,\E_\ell$. Let us consider an example. 

\begin{example}
Consider the Lie algebra 
\[ \fg=\{[a_{ij}]\in\Mat_5\,|\,a_{6-j,6-i}=-a_{ij},\text{ for all }i,j\}. \] 
(Note that $\fg$ consists of matrices that are `skew-symmetric about the other diagonal'.) It is easy to show that $\fg$ is isomorphic to $\fso_5$, so it is simple and has type $\mathrm{B}_2$. Choosing suitable $X,Y$, one finds that
\[
\cS_{\subreg}=\left\{\begin{bmatrix}
a&b&0&0&0\\
0&0&1&0&0\\
0&c&0&-1&0\\
d&0&-c&0&-b\\
0&-d&0&0&-a
\end{bmatrix}\right\},
\begin{array}{l} 
\chi_1=-a^2-2c,\\ 
\chi_2=2(a^2c+bd).
\end{array}
\]
So $\cS_{\subreg}\cap\chi^{-1}(u)$ is isomorphic to the following hypersurface in $\C^3$:
\[
\{(a,b,d)\in\C^3\,|\,a^4-2bd+u_1a^2+u_2=0\}.
\]
This family of hypersurfaces is a deformation of the Kleinian singularity of type $\A_3$, but compared with that of Example~\ref{ex:sl4-subreg}, it has only two parameters, not the full three of the semiuniversal deformation: in other words, there is no `$a$ term' in the equation. One can explain this deficiency in terms of the symmetry group $G_{X,Y}$ of $\cS_{\subreg}$, which in this case has two connected components. The identity component is isomorphic to $\C^\times$ and acts by fixing $a,c$ and scaling $b$ and $d$ by inverse scalars, analogously to the situation of Example~\ref{ex:sl4-subreg}. But there is also an element of the non-identity component that fixes $b,c,d$ and sends $a$ to $-a$, precluding the possibility of an `$a$ term' in the above equation.
\end{example}

Motivated by such considerations of symmetry, Slodowy realized that to extend Theorem~\ref{thm:brieskorn} to the other types of simple Lie algebras, one needs to assign to each type not a single finite subgroup of $SL_2(\C)$ but a pair of such subgroups $\Gamma,\Gamma'$, with $\Gamma$ being a normal subgroup of $\Gamma'$. The correct subgroups are specified in the following list:
\begin{itemize}
\item[$\mathrm{B}_\ell$:] $\Gamma$ of type $\A_{2\ell-1}$, $\Gamma'$ of type $\D_{\ell+2}$, $\Gamma'/\Gamma\cong S_2$;
\item[$\mathrm{C}_\ell$:] $\Gamma$ of type $\D_{\ell+1}$, $\Gamma'$ of type $\D_{2\ell}$, $\Gamma'/\Gamma\cong S_2$;
\item[$\mathrm{F}_4$:] $\Gamma$ of type $\E_6$, $\Gamma'$ of type $\E_7$, $\Gamma'/\Gamma\cong S_2$;
\item[$\mathrm{G}_2$:] $\Gamma$ of type $\D_4$, $\Gamma'$ of type $\E_7$, $\Gamma'/\Gamma\cong S_3$.
\end{itemize}
It is automatic that $\Gamma'/\Gamma$ acts on the Kleinian singularity $\C^2/\Gamma$.

\begin{theorem}[Slodowy~\cite{slod}] \label{thm:slodowy}
If $\fg$ is of type $\mathrm{B}_\ell$, $\mathrm{C}_\ell$, $\mathrm{F}_4$ or $\mathrm{G}_2$ and $\Gamma,\Gamma'\subset SL_2(\C)$ are as given above, then we have an isomorphism
\[
\cS_{\subreg}\cap\cN \cong \C^2/\Gamma,
\]
under which the action of $\Gamma'/\Gamma$ on $\C^2/\Gamma$ corresponds to the action of some subgroup of $G_{X,Y}$ on $\cS_{\subreg}\cap\cN$. Moreover, the family of varieties $\cS_{\subreg}\cap\chi^{-1}(u)$,  $u\in\C^\ell$, satisfies a suitable universal property among $(\Gamma'/\Gamma)$-equivariant deformations of $\C^2/\Gamma$.
\end{theorem}

One consequence of Theorems~\ref{thm:brieskorn} and~\ref{thm:slodowy} is that the isomorphism class of a simple Lie algebra $\fg$ can be determined purely by examining the surface singularity $\cS_{\subreg}\cap\cN$ and its deformation $\cS_{\subreg}\cap\chi^{-1}(u)$.

The universal property of the deformation $\cS_X\cap\chi^{-1}(u)$ has been generalized to non-subregular $X$ by Lehn--Namikawa--Sorger~\cite{lns}.

%-----------------
\section{Other minimal degenerations}
\label{sec:minsing}
%-----------------

In the previous section we considered the isolated singularity arising from the minimal degeneration $\cO_{\subreg}\prec\cO_{\reg}$. Further examples come from the other minimal degenerations of nilpotent orbits. Notably, at the other extreme of the closure order, we can consider $\{0\}\prec\cO_{\min}$. 

The singularity of $\overline{\cO_{\min}}=\cO_{\min}\cup\{0\}$ at $0$ is of a standard type. Let $\Pp(\fg)$ be the projective space whose points are the lines $L$ through $0$ in $\fg$. Those $L$ that lie in $\cO_{\min}$ form a closed subvariety $\Pp(\cO_{\min})$ of $\Pp(\fg)$. Being a projective variety with a homogeneous $G$-action, $\Pp(\cO_{\min})$ belongs to the class of \emph{partial flag varieties} for $G$. We have a tautological line bundle $Z\to\Pp(\cO_{\min})$, where the total space is defined by
\[
Z:=\{(X,L)\in\fg\times\Pp(\cO_{\min})\,|\,X\in L\}.
\]
Since $\Pp(\cO_{\min})$ is nonsingular, so is $Z$. The first projection $Z\to\overline{\cO_{\min}}$ is a resolution of singularities; it contracts the zero section of the line bundle $Z\to\Pp(\cO_{\min})$ to the single point $0$, while mapping the complement of this zero section isomorphically onto $\cO_{\min}$.

\begin{example} \label{ex:sln-minsing}
Take $\fg=\fsl_n$. Then $\cO_{\min}=\{X\in\fsl_n\,|\,\rk(X)=1\}$. A rank-$1$ matrix $X\in\Mat_n$ can be determined by specifying its image $V_1$ and kernel $V_{n-1}$ (subspaces of $\C^n$ of dimensions $1$ and $n-1$ respectively) as well as specifying the induced linear map $\C^n/V_{n-1}\to V_1$. The trace of $X$ is zero if and only if $V_1\subseteq V_{n-1}$. Hence $\Pp(\cO_{\min})$ can be identified with a variety of partial flags, where `partial flag' has the traditional sense of a chain of subspaces of $\C^n$:
\[
\Pp(\cO_{\min})\cong\{0\subset V_1\subseteq V_{n-1}\subset\C^n\,|\,\dim V_i=i\}.
\]
Then $Z\to\Pp(\cO_{\min})$ is identified with the line bundle over this partial flag variety where the fibre over $V_1\subseteq V_{n-1}$ is
\[
\{X\in\fsl_n\,|\,X(\C^n)\subseteq V_1,\, X(V_{n-1})=0\},
\]
and the singular variety $\overline{\cO_{\min}}$ is obtained from this line bundle by contracting the zero section to a point. Similar constructions of vector bundles over partial flag varieties will play a large role in the next section.
\end{example} 

The singularities arising from minimal degenerations of nilpotent orbits have now been completely described: in types $\A_\ell$, $\mathrm{B}_{\ell}$, $\mathrm{C}_{\ell}$ and $\D_\ell$ by Kraft--Procesi~\cite{kp2,kp3}, in type $\mathrm{G}_2$ by Kraft~\cite{kraft}, and in types $\E_\ell$ and $\mathrm{F}_4$ by Fu--Juteau--Levy--Sommers~\cite{fjls}.

The result for $\fsl_n$ (i.e.\ type $\A_{n-1}$) is particularly nice: in that case the singularities we have already considered, at the two extremes of the closure order, suffice to descibe all minimal degenerations. For the purpose of this statement, a \emph{minimal singularity} of type $\A_m$ means the singularity of $\overline{\cO_{\min}}$ at $0$ when $\fg=\fsl_{m+1}$, as discussed in Example~\ref{ex:sln-minsing}.
 
\begin{theorem}[Kraft--Procesi~\cite{kp2}] \label{thm:kp-mindeg}
Suppose $\fg=\fsl_n$ and $\cO_\mu\prec\cO_\lambda$. Let $\cS_\mu$ denote $\cS_X$ for some $X\in\cO_\mu$, so $\cS_\mu\cap\overline{\cO_{\lambda}}$ is a variety with an isolated singularity. Recall from Proposition~\ref{prop:sln-mindeg} that $\mu$ is obtained from $\lambda$ by moving a single corner box down and to the left.
\begin{itemize}
\item If the box moves one row down and $m$ columns left, then $\cS_\mu\cap\overline{\cO_{\lambda}}$ is isomorphic to a Kleinian singularity of type $\A_m$.
\item If the box moves $m$ rows down and one column left, then $\cS_\mu\cap\overline{\cO_{\lambda}}$ is isomorphic to a minimal singularity of type $\A_m$.
\end{itemize}
\end{theorem}
\noindent
The original statement of~\cite{kp2} was in terms of smooth equivalence of singularities; the promotion of this to an isomorphism of varieties will be explained after Corollary~\ref{cor:rowcol}.

\begin{example}
The following diagram shows the closure order on nilpotent orbits for $\fsl_6$ (i.e.\ type $\A_5$), with each orbit represented by the box-diagram of its partition. Every minimal degeneration is labelled by the type of the resulting isolated singularity, as given by Theorem~\ref{thm:kp-mindeg}. Following the notation of~\cite{kp2}, $\A_m$ means a Kleinian singularity of type $\A_m$ and $a_m$ means a minimal singularity of type $\A_m$. (Thus, the uppermost line is labelled $\A_5$ in accordance with Theorem~\ref{thm:brieskorn}, and the lowest line is labelled $a_5$ by definition.) Note that $a_1$ and $\A_1$ refer to the same thing, the nilpotent cone of $\fsl_2$ with its singularity at $0$.
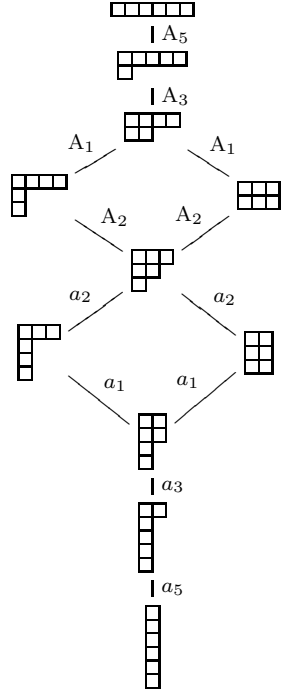
\setlength{\tabwidth}{1ex}
\setlength{\tabheight}{1ex}
\[
\xymatrix@R=6pt@C=10pt{
& *+{\begin{tableau}\row{\c\c\c\c\c\c}\end{tableau}} \ar@{-}^-{\A_5}[d] && \\
& *+{\begin{tableau}\row{\c\c\c\c\c}\row{\c}\end{tableau}} \ar@{-}^-{\A_3}[d] && \\
& *+{\begin{tableau}\row{\c\c\c\c}\row{\c\c}\end{tableau}} \ar@{-}_-{\A_1}[dl] \ar@{-}^-{\A_1}[dr] && \\
*+{\begin{tableau}\row{\c\c\c\c}\row{\c}\row{\c}\end{tableau}} \ar@{-}^-{\A_2}[dr] & &
*+{\begin{tableau}\row{\c\c\c}\row{\c\c\c}\end{tableau}} \ar@{-}_-{\A_2}[dl] & \\
& *+{\begin{tableau}\row{\c\c\c}\row{\c\c}\row{\c}\end{tableau}} \ar@{-}_-{a_2}[dl] 
\ar@{-}^-{a_2}[dr] && \\
*+{\begin{tableau}\row{\c\c\c}\row{\c}\row{\c}\row{\c}\end{tableau}} \ar@{-}^-{a_1}[dr] & &
*+{\begin{tableau}\row{\c\c}\row{\c\c}\row{\c\c}\end{tableau}} \ar@{-}_-{a_1}[dl] & \\
& *+{\begin{tableau}\row{\c\c}\row{\c\c}\row{\c}\row{\c}\end{tableau}} \ar@{-}^-{a_3}[d] && \\
& *+{\begin{tableau}\row{\c\c}\row{\c}\row{\c}\row{\c}\row{\c}\end{tableau}} 
\ar@{-}^-{a_5}[d] && \\
& *+{\begin{tableau}\row{\c}\row{\c}\row{\c}\row{\c}\row{\c}\row{\c}\end{tableau}}  & &
}
\]
\end{example}

\begin{example} \label{ex:g2}
When $\fg$ is of type $\mathrm{G}_2$, there are five nilpotent orbits and the closure order is a total order:
\[
\{0\}\prec\cO_{\min}\prec\cO\prec\cO_{\subreg}\prec\cO_{\reg}.
\]
The minimal-degeneration singularities are described in~\cite{kraft}.
The most interesting is that of the middle orbit closure $\overline{\cO}$ at points of $\cO_{\min}$: in~\cite{fjls} this is shown to be the non-normal isolated surface singularity defined as the image of
\[
\psi:\C^2\to\C^7:(t,u)\mapsto(t^2,tu,u^2,t^3,t^2u,tu^2,u^3).
\]
Note that $\psi$ is injective, so it is the normalization map of the singularity.
\end{example} 

%%%%%%%%%%%%%%%%%%%%%%%%%%%%%%%%%%%%%%%%%%%%%%%%%%%%%%%%%%%%%%%%%%%%%%%%%%%%%%%%%%%%%%%%%%

%-----------------
\section{The Springer resolution}
%-----------------

One of the nicest features of the singular varieties $\overline{\cO}$ is that they have resolutions that are vector bundles over partial flag varieties for $G$; see~\cite{panyushev} for a general construction. We have seen one example of this already, the resolution of $\overline{\cO_{\min}}$ described in Section~\ref{sec:minsing}. In this section we will examine a similar resolution of $\overline{\cO_{\reg}}=\cN$.

A \emph{Borel subalgebra} of $\fg$ is a subspace $\fb\subset\fg$, closed under the commutator bracket $[\cdot,\cdot]$, which is conjugate to a subspace of upper-triangular matrices and is maximal with these properties. The Borel subalgebras of $\fg$ are all in the same orbit for the adjoint action of $G$: in particular, they all have the same dimension, which turns out to be $\frac{1}{2}(\dim\fg+\ell)$. Moreover, the Borel subalgebras form a closed subvariety $\cB$ of the Grassmannian of all subspaces of $\fg$ of this dimension. This projective homogeneous $G$-variety $\cB$ is called the \emph{flag variety} of $G$ (or of $\fg$). The name is explained by the following example.
 
\begin{example}
When $\fg=\fsl_n$, the Borel subalgebras are exactly the conjugates of the subalgebra of upper-triangular matrices. Since these are the stabilizing subalgebras of `complete flags' in $\C^n$ (i.e.\ chains of subspaces of $\C^n$, one of each dimension), the variety $\cB$ for $\fsl_n$ can be identified with the variety of such flags,
\[
\cF:=\{0\subset V_1\subset V_2\subset \cdots \subset V_{n-1}\subset\C^n\,|\,\dim V_i=i\}.
\] 
We will make this identification $\cB\cong\cF$ whenever we consider $\fsl_n$.
\end{example} 

We have a tautological vector bundle $\tfg\to\cB$, where the total space is defined by 
\[
\tfg=\{(X,\fb)\in\fg\times\cB\,|\,X\in\fb\}.
\]
Let $\pi:\tfg\to\fg$ denote the first projection. For any $X\in\fg$, the fibre $\pi^{-1}(X)$ can be identified with the variety of Borel subalgebras containing $X$; this variety is called the \emph{Springer fibre} of $X$.
The map $\pi$ as a whole is known as the \emph{Grothendieck--Springer simultaneous resolution}, because it simultaneously provides resolutions of every fibre of the adjoint quotient map $\chi:\fg\to\C^\ell$.

\begin{theorem}[Grothendieck, Springer, see~\cite{slod}]
For any $u\in\C^\ell$, each connected component of $\pi^{-1}(\chi^{-1}(u))$ is a resolution of $\chi^{-1}(u)$. In particular, $\tcN:=\pi^{-1}(\cN)$ \textup{(}which is connected\textup{)} is a resolution of $\cN$.
\end{theorem}
\noindent
Here the resolution maps are the restrictions of $\pi$ to the stated domains. 

The reason that $\tcN$ is connected is that the nilpotent elements of a Borel subalgebra $\fb$ form a vector subspace $\fn_\fb$, the \emph{nilradical} of $\fb$. Thus
\[ 
\tcN = \{(X,\fb)\in\fg\times\cB\,|\,X\in\fn_\fb\}\to\cB
\]
is a sub-vector bundle of the vector bundle $\tfg\to\cB$. The resolution $\tcN\to\cN$ is called the \emph{Springer resolution} of the nilpotent cone. 

\begin{example} \label{ex:sln-springer}
For $\fg=\fsl_n$, the Grothendieck--Springer simultaneous resolution can be described thus:
\[
%\begin{split}
%\pi^{-1}(X)&\cong\{(V_i)\in\cF\,|\,X(V_i)\subseteq V_i\text{ for all }i\},\\
\tfg\cong\{(X,(V_i))\in\fsl_n\times\cF\,|\,X(V_i)\subseteq V_{i}\text{ for all }i\}.
%\end{split}
\]
Notice that if $X$ stabilizes a complete flag $(V_i)$ in the sense of this condition that $X(V_i)\subseteq V_i$ for all $i$, then $X$ is nilpotent if and only if we have the stronger condition $X(V_i)\subseteq V_{i-1}$ for all $i$, where we set $V_0=0$ and $V_n=\C^n$ to take care of the $i=1$ and $i=n$ cases. (In other words, the nilradical of the Borel subalgebra of upper-triangular matrices consists of the \emph{strictly} upper-triangular matrices.) Hence
\[
\tcN\cong\{(X,(V_i))\in\fsl_n\times\cF\,|\,X(V_i)\subseteq V_{i-1}\text{ for all }i\},
\]
and for $X\in\cN$ the Springer fibre has the following description:
\[
\pi^{-1}(X)\cong\{(V_i)\in\cF\,|\,X(V_i)\subseteq V_{i-1}\text{ for all }i\}.
\]
Let us consider these Springer fibres over some particular nilpotent orbits of $\fsl_n$. If $X=0$ we of course have $\pi^{-1}(0)\cong\cF$. At the other extreme, if $X\in\cO_{\reg}$ has a single Jordan block, then there is a unique flag $(V_i)\in\cF$ such that $X(V_i)\subseteq V_{i-1}$ for all $i$, namely the one given by $V_i=\ker(X^i)$, so $\pi^{-1}(X)$ is a single point (in general, the Springer resolution $\tcN\to\cN$ is an isomorphism over $\cO_{\reg}$). If $X\in\cO_{\subreg}=\cO_{(n-1,1)}$, one can easily show that $\pi^{-1}(X)$ is the union of the following $n-1$ projective lines, each of which meets the adajcent lines transversely at a single point and does not meet any of the other lines:
\[
\begin{split}
&\{(V_i)\,|\,V_2=\ker(X),V_3=\ker(X^2),V_4=\ker(X^3),\cdots\},\\
&\{(V_i)\,|\,V_1=\im(X^{n-2}),V_3=\ker(X^2),V_4=\ker(X^3),\cdots\},\\
&\{(V_i)\,|\,V_1=\im(X^{n-2}),V_2=\im(X^{n-3}),V_4=\ker(X^3),\cdots\},\\
&\qquad\vdots\\
&\{(V_i)\,|\,V_1=\im(X^{n-2}),V_2=\im(X^{n-3}),\cdots,V_{n-2}=\im(X)\}.
\end{split}
\]
The geometry of general Springer fibres can be very complicated. 
\end{example}

Slodowy showed in~\cite{slod} that the Grothendieck--Springer simultaneous resolution theorem implies an analogous statement for each slice $\cS_X$ for $X\in\fg$. In particular, the Springer resolution restricts to a resolution $\pi^{-1}(\cS_X\cap\cN)\to\cS_X\cap\cN$. This recovers a much-studied resolution of the Kleinian singularities: 

\begin{theorem}[Brieskorn, Slodowy~\cite{slod}]
If $\Gamma$ is the finite subgroup of $SL_2(\C)$ such that $\cS_{\subreg}\cap\cN\cong\C^2/\Gamma$, then  the resolution 
\[ \pi^{-1}(\cS_{\subreg}\cap\cN)\to \cS_{\subreg}\cap\cN \]
corresponds to the minimal resolution of $\C^2/\Gamma$.
\end{theorem}

\begin{example}
Take $\fg=\fsl_3$. Recall the explicit matrix description of an element of $\cS_{\subreg}\cap\cN$ from Example~\ref{ex:sl3-subreg}. After some simple computations of the conditions for a flag $0\subset V_1\subset V_2\subset\C^3$ to belong to the Springer fibre of that matrix, one concludes that
\[
\begin{split}
\pi^{-1}(\cS_{\subreg}\cap\cN)\cong
\{((a,c,d)&,[s:t],[u:v])\in\C^3\times\Pp^1\times\Pp^1\,|\,\\
%&8a^3=cd,\\
&\begin{bmatrix}-4a^2&c\\d&-2a\end{bmatrix}\begin{bmatrix}s\\t\end{bmatrix}=\begin{bmatrix}0\\0\end{bmatrix},\\
&\begin{bmatrix}u\quad v\end{bmatrix}\begin{bmatrix}-4a^2&c\\d&-2a\end{bmatrix}=\begin{bmatrix}0\quad 0\end{bmatrix},\\
&2asu=tv\}.
\end{split}
\]
Here the resolution map is the projection onto the $(a,c,d)$ component, which necessarily lies in the type-$\A_2$ Kleinian singularity defined by the equation $8a^3=cd$.
If $(a,c,d)$ is not the singular point $(0,0,0)$, then $[s:t],[u:v]$ are uniquely determined by the above conditions; if $(a,c,d)=(0,0,0)$, then the final condition becomes $tv=0$, so the fibre is a union of two projective lines meeting transversely at a point (this is a special case of the subregular Springer fibre described in Example~\ref{ex:sln-springer}).  
\end{example}

%-----------------
\section{Normality of orbit closures for $\fsl_n$}
%-----------------

In this section we will take $\fg=\fsl_n$ and discuss some special properties of the orbit closures $\overline{\cO_\lambda}$ in this case, in particular the proof by Kraft and Procesi that they are normal varieties. As seen in Example~\ref{ex:g2}, there do exist non-normal nilpotent orbit closures in simple Lie algebras of other types; see~\cite{broer,kraft,kp3,sommers1,sommers2} for results on this. 

We can generalize the Springer resolution of $\cN$ to construct a resolution of $\overline{\cO_\lambda}$ where $\lambda$ is any partition of $n$. For this, let $r_1,\cdots,r_m$ be the lengths of the columns of the box-diagram of $\lambda$, in decreasing order (thus the number $m$ is by definition equal to $\lambda_1$). We consider the partial flag variety
\[
\cF_\lambda=\left\{0=U_0\subset U_1\subset \cdots \subset U_m=\C^n\,\left|\,\dim\frac{U_i}{U_{i-1}}=r_i\right.\right\},
\]
and the vector bundle over $\cF_\lambda$ whose total space is
\[
\tcF_\lambda=\{(X,(U_i))\in\fsl_n\times\cF_\lambda\,|\,X(U_i)\subseteq U_{i-1}\text{ for all }i\}.
\]
Then it is not hard to show that the first projection maps $\tcF_\lambda$ to $\overline{\cO_\lambda}$, and in fact is a resolution $\pi_\lambda:\tcF_\lambda\to\overline{\cO_\lambda}$. If $\lambda=(n)$, then $m=n$ and $r_i=1$ for all $i$, so we are dealing with complete flags and we recover the Springer resolution. 

\begin{example} \label{ex:sln-minres}
If $\lambda=(2,1,\cdots,1)$, we have $m=2$, $r_1=n-1$ and $r_2=1$, so $\cF_{(2,1,\cdots,1)}$ is the projective space of $(n-1)$-dimensional subspaces of $\C^n$, and $\tcF_{(2,1,\cdots,1)}$ is the rank-$(n-1)$ vector bundle over this projective space where the fibre over the $(n-1)$-dimensional subspace $U_1$ is 
\[
\{X\in\fsl_n\,|\,X(\C^n)\subseteq U_1,\,X(U_1)=0\}.
\]
The resolution $\pi_{(2,1,\cdots,1)}:\tcF_{(2,1,\cdots,1)}\to\overline{\cO_{(2,1,\cdots,1)}}$ contracts the zero section of the vector bundle to the singular point $0$ and maps the complement isomorphically onto $\cO_{(2,1,\cdots,1)}$; for $X\in\cO_{(2,1,\cdots,1)}$, the unique $(n-1)$-dimensional subspace $U_1$ satisfying the above conditions is $\ker(X)$. Note that $\pi_{(2,1,\cdots,1)}$ is different from the resolution of $\overline{\cO_{(2,1,\cdots,1)}}$ considered in Example~\ref{ex:sln-minsing}: it is more economical in the sense that the fibre over $0$ is a smaller-dimensional partial flag variety, but less canonical in the sense that it breaks the symmetry between dimension-$1$ and codimension-$1$ subspaces.
\end{example}

There is a connection between these resolutions and \emph{quiver representations}, which really just means diagrams of linear maps between vector spaces. A pair $(X,(U_i))\in\tcF_\lambda$ can be encoded as such a diagram:
\[
\vcenter{
\xymatrix@R=15pt{
\C^n\ar@/^/[dd]^-{\pr}&&&
\\
\\
V_1\ar@/^/[r]^-{\pr}\ar@/^/[uu]^-{X}&V_2\ar@/^/[l]^-{X}\ar@/^/[r]^-{\pr}&V_3\ar@/^/[l]^-{X}\ar@/^/[r]^-{\pr}&\cdots\ar@/^/[r]^-{\pr}\ar@/^/[l]^-{X}&V_{m-2}\ar@/^/[r]^-{\pr}\ar@/^/[l]^-{X}&V_{m-1}\ar@/^/[l]^-{X}
}
}
\]
where $V_i=\C^n/U_i$, $\pr:V_i\to V_{i+1}$ is the canonical projection, and we abuse notation by writing $X$ for the map $V_{i+1}\to V_i$ induced by $X$. Note that here we have $\dim V_i=n-r_1-\cdots-r_i$; we denote this quantity by $v_i$, and let $\bv$ denote the $(m-1)$-tuple $(v_1,\cdots,v_{m-1})$. 

Consider the affine variety $\Lambda_{\bv,n}$ of all diagrams of linear maps
\[
\vcenter{
\xymatrix@R=15pt{
\C^n\ar@/^/[dd]^-{A_0}&&&
\\
\\
\C^{v_1}\ar@/^/[r]^-{A_1}\ar@/^/[uu]^-{B_0}&\C^{v_2}\ar@/^/[l]^-{B_1}\ar@/^/[r]^-{A_2}&\C^{v_3}\ar@/^/[l]^-{B_2}\ar@/^/[r]^-{A_3}&\cdots\ar@/^/[r]^-{A_{m-3}}\ar@/^/[l]^-{B_{3}}&\C^{v_{m-2}}\ar@/^/[r]^-{A_{m-2}}\ar@/^/[l]^-{B_{m-3}}&\C^{v_{m-1}}\ar@/^/[l]^-{B_{m-2}}
}
}
\]
satisfying the equations $A_{i-1}B_{i-1}=B_{i}A_{i}$ for $1\leq i\leq m-2$ and $A_{m-2}B_{m-2}=0$. This is a closed subvariety of the affine space 
\[ \Hom(\C^n,\C^{v_1})\oplus\Hom(\C^{v_1},\C^n)\oplus\Hom(\C^{v_1},\C^{v_2})\oplus\Hom(\C^{v_2},\C^{v_1})\oplus\cdots. \] 
It is almost, but not exactly, correct to say that the diagram defined by $(X,(U_i))\in\tcF_\lambda$ represents a point of $\Lambda_{\bv,n}$: certainly the equations are satisfied, but one has to make some choice of identifications of the vector space $V_i$ with $\C^{v_i}$ for all $i$, and this results in a change-of-basis indeterminacy. To take this into account, we need to consider the natural action of the group $G_{\bv}:=GL_{v_1}\times\cdots\times GL_{v_{m-1}}$ on $\Lambda_{\bv,n}$. This action is free on 
the open subset 
\[ \Lambda_{\bv,n}^s:=\{(A_i,B_i)\in\Lambda_{\bv,n}\,|\,A_i\text{ surjective for all }i\}, \]
so there is a well-defined geometric quotient variety $\Lambda_{\bv,n}^s/G_\bv$, and the correct statement is that we have an isomorphism 
\[
\Lambda_{\bv,n}^s/G_\bv\simto \tcF_\lambda:(A_i,B_i)\mapsto(B_0A_0,(\ker(A_{i-1}\cdots A_1A_0))).
\]
Here and below, when we want to specify a map whose domain is a quotient, we write the formula for the composition with the quotient projection.

It is now natural to consider the quotient $\Lambda_{\bv,n}/G_\bv$, which by definition is the affine variety whose algebra of functions is the invariant ring $\C[\Lambda_{\bv,n}]^{G_\bv}$.

\begin{theorem}[Kraft--Procesi~\cite{kp1}] \label{thm:normality}
We have an isomorphism
\[
\Lambda_{\bv,n}/G_\bv\simto \overline{\cO_{\lambda}}:(A_i,B_i)\mapsto B_0A_0.
\]
Moreover, $\Lambda_{\bv,n}$ is a normal variety and hence so is $\overline{\cO_{\lambda}}$.
\end{theorem}

\noindent
Kraft and Procesi could prove that $\Lambda_{\bv,n}$ is normal by proving first that it is a complete intersection (not true of $\overline{\cO_{\lambda}}$, as we saw in Example~\ref{ex:sl3-orbits}). It is easy to see that the property of normality is inherited by quotients under group actions.

Note that the open embedding $\Lambda_{\bv,n}^s\to\Lambda_{\bv,n}$ induces a map on quotients $\Lambda_{\bv,n}^s/G_\bv\to\Lambda_{\bv,n}/G_\bv$ that is far from injective. Indeed, under the above isomorphisms this map $\Lambda_{\bv,n}^s/G_\bv\to\Lambda_{\bv,n}/G_\bv$ corresponds to the resolution map $\tcF_\lambda\to\overline{\cO_{\lambda}}$. To understand this, one has to recall that the points of $\Lambda_{\bv,n}/G_\bv$ correspond to the \emph{closed} $G_\bv$-orbits in $\Lambda_{\bv,n}$. Every $G_\bv$-orbit in $\Lambda_{\bv,n}^s$ is closed in $\Lambda_{\bv,n}^s$, but in general it is not closed in $\Lambda_{\bv,n}$; it gets mapped to the unique closed orbit in its closure in $\Lambda_{\bv,n}$. Compare the map from $(\C^n\setminus\{0\})/\C^\times\cong\Pp^{n-1}$ to $\C^n/\C^\times$; the latter is a single point because the only closed $\C^\times$-orbit in $\C^n$ is $\{0\}$.

\begin{example} 
Continue with Example~\ref{ex:sln-minres}. Since $v_1=1$, we have
\[
\begin{split}
\Lambda_{\bv,n}&=\{(A_0,B_0)\in\Hom(\C^n,\C^1)\times\Hom(\C^1,\C^n)\,|\,A_0B_0=0\}\\
&\cong\{((a_1,\cdots,a_n),(b_1,\cdots,b_n))\in(\C^n)^2\,|\,\sum_{i=1}^n a_ib_i=0\},
\end{split}
\]
with $G_\bv=\C^\times$ scaling the vectors $(a_1,\cdots,a_n)$ and $(b_1,\cdots,b_n)$ by inverse scalars. In this example, Theorem~\ref{thm:normality} is easy to see: the $G_\bv$-invariant polynomial functions on $\Lambda_{\bv,n}$ are generated by the functions $((a_1,\cdots,a_n),(b_1,\cdots,b_n))\mapsto a_j b_i$ as $i$ and $j$ vary, and these functions form the entries of a map $\Lambda_{\bv,n}\to\overline{\cO_{(2,1,\cdots,1)}}$, factoring through the isomorphism of Theorem~\ref{thm:normality}. Note that
\[
\Lambda_{\bv,n}^s=\{(A_0,B_0)\in\Lambda_{\bv,n}\,|\,A_0\neq 0\},
\] 
so $\Lambda_{\bv,n}^s/G_\bv$ has a map to $\Pp^{n-1}$ defined by $(A_0,B_0)\mapsto [a_1:\cdots:a_n]$; this is a rank-$(n-1)$ vector bundle, corresponding to the bundle $\tcF_{(2,1,\cdots,1)}\to\cF_{(2,1,\cdots,1)}$ seen in Example~\ref{ex:sln-minres}. 
\end{example}

For the purposes of the next section it is useful to note a slight generalization of the above results. In the definition of $r_1,\cdots,r_m$ (and hence of $\bv$), we can actually take the column-lengths of $\lambda$ in any order, not necessarily in decreasing order; what is more, we can also include some zeroes among the $r_i$'s (meaning $m$ could be larger than $\lambda_1$). In this more flexible setting, $r_1,\cdots,r_m$ are no longer uniquely determined by $\lambda$, but we will continue to use the notations $\cF_\lambda$ and $\tcF_\lambda$ for the varieties defined using the chosen $r_1,\cdots,r_m$. It is still true that $\pi_\lambda:\tcF_{\lambda}\to\overline{\cO_{\lambda}}$ is a resolution, and we still have the same isomorphism $\Lambda_{\bv,n}^s/G_\bv\simto \tcF_{\lambda}$. Theorem~\ref{thm:normality} requires a modification, because $\Lambda_{\bv,n}$ may now not even be irreducible, let alone normal; however, if we let $\Lambda_{\bv,n}^1$ be the closure of $\Lambda_{\bv,n}^s$ in $\Lambda_{\bv,n}$ (an irreducible component of $\Lambda_{\bv,n}$), then we still have an isomorphism $\Lambda_{\bv,n}^1/G_\bv\simto \overline{\cO_{\lambda}}$. This all follows from the more general Theorem~\ref{thm:maffei} below; see~\cite{shmelkin} for a description of $\Lambda_{\bv,n}/G_\bv$ in the present setting.

\begin{example}
Take $\fg=\fsl_n$ and $\lambda=(2,1,\cdots,1)$. Instead of the choice used in Example~\ref{ex:sln-minres} we could set $r_1=1$, $r_2=n-1$. Then 
\[
\Lambda_{\bv,n}=\{(A_0,B_0)\in\Hom(\C^n,\C^{n-1})\times\Hom(\C^{n-1},\C^n)\,|\,A_0B_0=0\}
\]
which has multiple irreducible components (for $n\geq 4$). The irreducible component $\Lambda_{\bv,n}^1$ is defined by the extra condition $\rk(B_0)\leq 1$.
\end{example}

%-----------------
\section{Maffei's theorem and its consequences}
\label{sec:maffei}
%-----------------

Maffei's theorem~\cite{maffei} is a generalization of the isomorphism statements in the previous section, that applies to the varieties $\cS_\mu\cap\overline{\cO_{\lambda}}$, where $\mu\trianglelefteq\lambda$ and $\cS_\mu$ means $\cS_X$ for some $X\in\cO_\mu$. We will refer to the point $X$ as the \emph{base-point} of $\cS_\mu\cap\overline{\cO_{\lambda}}$.

As at the end of the previous section, we let $r_1,\cdots,r_m$ be the column-lengths of $\lambda$ in some order and possibly with some zeroes included, allowing us to assume $m>\mu_1$. We define the resolution $\pi_\lambda:\tcF_\lambda\to\overline{\cO_\lambda}$ as before. Let $w_i$ be the multiplicity of $i$ in $\mu$ for $1\leq i\leq m-1$, and let $\bw=(w_1,\cdots,w_{m-1})$. The new definition of $\bv=(v_1,\cdots,v_{m-1})$ is
\[
v_i:=w_1+2w_2+\cdots+iw_i+iw_{i+1}+\cdots+iw_{m-1}-r_1-\cdots-r_i.
\]
Because of our assumptions that $\mu\trianglelefteq\lambda$ and $m>\mu_1$, the quantity $w_1+2w_2+\cdots+iw_i+iw_{i+1}+\cdots+iw_{m-1}$ appearing here is the sum of the lengths of the $i$ longest columns of $\mu$, and $v_i\geq 0$. Note that if $\mu=(1,\cdots,1)$, then $\bw=(n,0,\cdots,0)$ and $\bv$ is the same as in the previous section.

Now we consider the affine variety $\Lambda_{\bv,\bw}$ of all diagram of linear maps
\[
\vcenter{
\xymatrix@R=15pt{
\C^{w_1}\ar@/^/[dd]^-{\Gamma_1}&\C^{w_2}\ar@/^/[dd]^-{\Gamma_2}&\C^{w_3}\ar@/^/[dd]^-{\Gamma_3}&&\C^{w_{m-2}}\ar@/^/[dd]^-{\Gamma_{m-2}}&\C^{w_{m-1}}\ar@/^/[dd]^-{\Gamma_{m-1}}
\\
\\
\C^{v_1}\ar@/^/[r]^-{A_1}\ar@/^/[uu]^-{\Delta_1}&\C^{v_2}\ar@/^/[l]^-{B_1}\ar@/^/[r]^-{A_2}\ar@/^/[uu]^-{\Delta_2}&\C^{v_3}\ar@/^/[l]^-{B_2}\ar@/^/[r]^-{A_3}\ar@/^/[uu]^-{\Delta_3}&\cdots\ar@/^/[r]^-{A_{m-3}}\ar@/^/[l]^-{B_{3}}&\C^{v_{m-2}}\ar@/^/[r]^-{A_{m-2}}\ar@/^/[l]^-{B_{m-3}}\ar@/^/[uu]^-{\Delta_{m-2}}&\C^{v_{m-1}}\ar@/^/[l]^-{B_{m-2}}\ar@/^/[uu]^(.6){\Delta_{m-1}}
}
}
\]
satisfying the equations $A_{i-1}B_{i-1}+\Gamma_i\Delta_i=B_{i}A_{i}$ for $1\leq i\leq m-1$, where we interpret $A_0,B_0,A_{m-1},B_{m-1}$ as $0$. Again we have a natural action of $G_\bv$ on $\Lambda_{\bv,\bw}$. If $\mu=(1,\cdots,1)$, then we can clearly identify $\Lambda_{\bv,\bw}$ with the previous $\Lambda_{\bv,n}$, with $\Gamma_1$ and $\Delta_1$ becoming $A_0$ and $B_0$ respectively.

The generalization of $\Lambda_{\bv,n}^s$ is the open subset $\Lambda_{\bv,\bw}^s$ of $\Lambda_{\bv,\bw}$ defined by the `stability' condition that there is no proper subspace of $\bigoplus_i\C^{v_i}$ that contains $\bigoplus_i\im(\Gamma_i)$ and is stable under all the maps $A_i,B_i$. Let $\Lambda_{\bv,\bw}^1$ be the closure of $\Lambda_{\bv,\bw}^s$ in $\Lambda_{\bv,\bw}$. It follows from the next result that $\Lambda_{\bv,\bw}^s$ is nonempty, and it is clearly stable under simultaneous scaling of all the maps $A_i,B_i,\Gamma_i,\Delta_i$, so the point $0\in\Lambda_{\bv,\bw}$ belongs to $\Lambda_{\bv,\bw}^1$.

\begin{theorem}[Maffei~{\cite[Theorem 8]{maffei}}] \label{thm:maffei}
With notation as above, we have variety isomorphisms
\[
\begin{split}
\Lambda_{\bv,\bw}^s/G_\bv&\cong \pi_\lambda^{-1}(\cS_\mu\cap\overline{\cO_{\lambda}}),\\
\Lambda_{\bv,\bw}^1/G_\bv&\cong \cS_\mu\cap\overline{\cO_{\lambda}},
\end{split}
\]
under which the map $\Lambda_{\bv,\bw}^s/G_\bv\to\Lambda_{\bv,\bw}^1/G_\bv$ induced by the embedding $\Lambda_{\bv,\bw}^s\to\Lambda_{\bv,\bw}^1$ corresponds to the restriction of $\pi_\lambda$ to $\pi_\lambda^{-1}(\cS_\mu\cap\overline{\cO_{\lambda}})$. The $G_\bv$-orbit of $0\in\Lambda_{\bv,\bw}^1$ corresponds to the base-point of $\cS_\mu\cap\overline{\cO_{\lambda}}$.
\end{theorem}

The varieties $\Lambda_{\bv,\bw}^s/G_\bv$ and $\Lambda_{\bv,\bw}^1/G_\bv$ are examples of \emph{Nakajima quiver varieties} (see~\cite{ginzburg,nak1,nak2}), and these isomorphisms were originally conjectured by Nakajima in~\cite{nak1}. Maffei's definition of the isomorphisms is effectively inductive, and as a result much less explicit than in the $\mu=(1,\cdots,1)$ case considered by Kraft and Procesi in~\cite{kp1}.

\begin{example}
Take $\fg=\fsl_3$, $\mu=(2,1)$, and $\lambda=(3)$. Setting $r_1=r_2=r_3=1$, we get $\bv=\bw=(1,1)$.
Hence $\Lambda_{\bv,\bw}$ consists of diagrams
\[
\vcenter{
\xymatrix@R=10pt{
\C^{1}\ar@/^/[dd]^-{\Gamma_1}&\C^{1}\ar@/^/[dd]^-{\Gamma_2}
\\
\\
\C^{1}\ar@/^/[r]^-{A_1}\ar@/^/[uu]^-{\Delta_1}&\C^{1}\ar@/^/[l]^-{B_1}\ar@/^/[uu]^-{\Delta_2}
}
}
\]
with $\Gamma_1\Delta_1=B_1 A_1$, $A_1B_1=-\Gamma_2\Delta_2$. Since all the vector spaces are $1$-dimensional, $\Gamma_i,\Delta_i,A_1,B_1$ themselves constitute functions on $\Lambda_{\bv,\bw}$. It is easy to see that the invariant ring $\C[\Lambda_{\bv,\bw}]^{G_\bv}$ is generated by
\[
a=\Delta_1\Gamma_1=-\Delta_2\Gamma_2,\ b=\Delta_1 B_1\Gamma_2,\ c=\Delta_2 A_1\Gamma_1,
\]
which satisfy the single equation $a^3+bc=0$. Hence $\Lambda_{\bv,\bw}/G_\bv$ is isomorphic to the type-$\A_2$ Kleinian singularity, in accordance with Theorem~\ref{thm:maffei} and Example~\ref{ex:sl3-subreg}. (In this case, $\Lambda_{\bv,\bw}=\Lambda^1_{\bv,\bw}$.) 
\end{example}

The Kleinian singularities of all types $\A_\ell$, $\D_\ell$ and $\E_\ell$ have a uniform construction in terms of Nakajima quiver varieties of the corresponding type~\cite{cass-slod,kronheimer}. Hence the subregular slice $\cS_{\subreg}\cap\cN$ in any simple Lie algebra $\fg$ is isomorphic to a certain Nakajima quiver variety. A partial generalization of this statement to other varieties $\cS_X\cap\cN$ is given in~\cite{hendersonlicata}. 

Returning to type $\A$, Maffei's description of the varieties $\cS_\mu\cap\overline{\cO_{\lambda}}$ in terms of quiver varieties has some easy consequences that relate Slodowy slices in different nilpotent cones. For convenience in stating these consequences, we identify partitions of $n$ with their box diagrams.
\begin{corollary} \label{cor:rowcol}
If $\lambda$ and $\mu$ have the same first row or the same first column, then we have a base-point-preserving isomorphism $\cS_\mu\cap\overline{\cO_{\lambda}}\cong\cS_{\mu'}\cap\overline{\cO_{\lambda'}}$ where $\lambda'$,$\mu'$ are obtained from $\lambda$,$\mu$ by deleting that row/column.
\end{corollary}
\begin{proof}
Suppose $\lambda$ and $\mu$ have the same first row, i.e.\ $\lambda_1=\mu_1$. We can take $m=\lambda_1+1$, and let $r_1,\cdots,r_{m-1}$ be the column-lengths of $\lambda$ in decreasing order, with $r_m=0$. Then $v_{m-1}=0$ and $w_{m-1}>0$ by definition. When we form the $(m-1)$-tuples $\bv'$ and $\bw'$ for $\lambda'$ and $\mu'$ in the same way, we find that $\bv'=\bv$ and $\bw'$ differs from $\bw$ only in that $w_{m-1}'=w_{m-1}-1$. We have an obvious $G_\bv$-equivariant isomorphism $\Lambda_{\bv,\bw}\simto\Lambda_{\bv,\bw'}$, because the fact that $v_{m-1}=0$ forces the maps $\Gamma_{m-1},\Delta_{m-1}$ to be zero whether one works with $w_{m-1}$ or $w_{m-1}'$. This induces an isomorphism $\Lambda_{\bv,\bw}^1/G_\bv\simto\Lambda_{\bv,\bw'}^1/G_\bv$. So Theorem~\ref{thm:maffei} implies the desired isomorphism $\cS_\mu\cap\overline{\cO_{\lambda}}\cong\cS_{\mu'}\cap\overline{\cO_{\lambda'}}$. The proof when $\lambda$ and $\mu$ have the same first column is similar, but with $v_1=0$ rather than $v_{m-1}=0$.     
\end{proof}

If $\cO_\mu\prec\cO_\lambda$ is a minimal degeneration as described in Proposition~\ref{prop:sln-mindeg}, then one can apply Corollary~\ref{cor:rowcol} repeatedly to remove all the common rows and columns, leaving a minimal degeneration $\cO_{\mu^*}\prec\cO_{\lambda^*}$ where either $\mu^*=(1,\cdots,1)$ or $\lambda^*$ has a single part. This proves Theorem~\ref{thm:kp-mindeg}. Well before Theorem~\ref{thm:maffei}, Kraft and Procesi proved in~\cite{kp2} a slightly weaker form of Corollary~\ref{cor:rowcol}, with smooth equivalence instead of isomorphism; this is how they proved their version of Theorem~\ref{thm:kp-mindeg}. 

In view of Theorem~\ref{thm:juteau-mautner}, either form of Corollary~\ref{cor:rowcol} immediately implies the following known result in representation theory, as was observed by Juteau~\cite{juteau}:

\begin{corollary}[James~\cite{james:iii}]
Let $p$ be a prime. If $\lambda$ and $\mu$ have the same first row or the same first column, then $d_{\lambda\mu}^p=d_{\lambda'\mu'}^p$ where $\lambda'$,$\mu'$ are obtained from $\lambda$,$\mu$ by deleting that row/column.
\end{corollary}

James' result was generalized by Donkin~\cite{donkin} to the case where $\lambda$ and $\mu$ admit a `horizontal cut'. This more general result can be deduced from an analogous generalization of Corollary~\ref{cor:rowcol} (the corresponding quiver varieties have $v_i=0$ for some $i$ not necessarily in $\{1,m-1\}$). 

Another consequence of Theorem~\ref{thm:maffei} is:
\begin{corollary} \label{cor:rotate}
Suppose that $\lambda$ and $\mu$ both have at most $t$ nonzero parts, each of which is at most $m$. 
Then we have a base-point-preserving isomorphism $\cS_\mu\cap\overline{\cO_{\lambda}}\cong\cS_{\mu^c}\cap\overline{\cO_{\lambda^c}}$ where $\lambda^c$,$\mu^c$ are obtained from $\lambda$,$\mu$ respectively by taking complements in a $t\times m$ rectangle \textup{(}and rotating through $180^\circ$ to obtain box-diagrams of the normal orientation\textup{)}.
\end{corollary}

\begin{example}
If $\lambda=(5,4,4,3)$ and $\mu=(3,3,3,3,2,2)$, then we can take $t=6$ and $m=5$, producing $\lambda^c=(5,5,2,1,1)$ and $\mu^c=(3,3,2,2,2,2)$. In the following picture, the boxes of $\lambda^c,\mu^c$ are those containing dots.
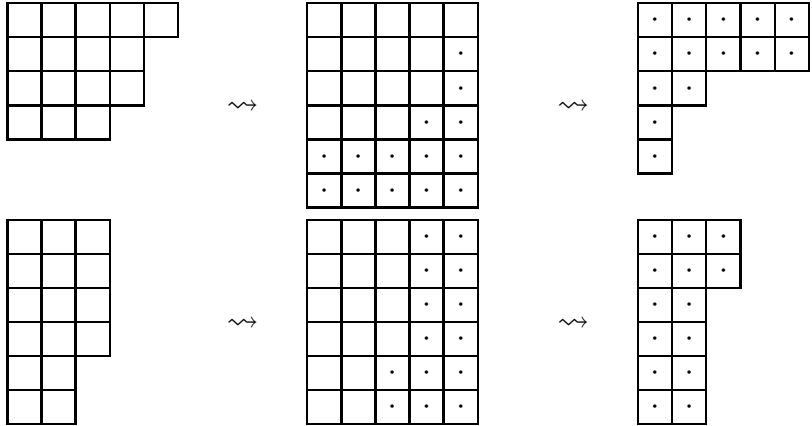
\[
\begin{split}
\begin{tableau}
\row{\c\c\c\c\c}
\row{\c\c\c\c}
\row{\c\c\c\c}
\row{\c\c\c}
\row{\genblankbox{1}{1}}
\row{\genblankbox{1}{1}}
\end{tableau}
\quad\rightsquigarrow\quad
\begin{tableau}
\row{\c\c\c\c\c}
\row{\c\c\c\c\ghost}
\row{\c\c\c\c\ghost}
\row{\c\c\c\ghost\ghost}
\row{\ghost\ghost\ghost\ghost\ghost}
\row{\ghost\ghost\ghost\ghost\ghost}
\end{tableau}
\qquad&\rightsquigarrow\quad
\begin{tableau}
\row{\ghost\ghost\ghost\ghost\ghost}
\row{\ghost\ghost\ghost\ghost\ghost}
\row{\ghost\ghost}
\row{\ghost}
\row{\ghost}
\row{\genblankbox{1}{1}}
\end{tableau}
\\
\begin{tableau}
\row{\c\c\c\genblankbox{1}{1}\genblankbox{1}{1}}
\row{\c\c\c}
\row{\c\c\c}
\row{\c\c\c}
\row{\c\c}
\row{\c\c}
\end{tableau}
\quad\rightsquigarrow\quad
\begin{tableau}
\row{\c\c\c\ghost\ghost}
\row{\c\c\c\ghost\ghost}
\row{\c\c\c\ghost\ghost}
\row{\c\c\c\ghost\ghost}
\row{\c\c\ghost\ghost\ghost}
\row{\c\c\ghost\ghost\ghost}
\end{tableau}
\qquad&\rightsquigarrow\quad
\begin{tableau}
\row{\ghost\ghost\ghost}
\row{\ghost\ghost\ghost}
\row{\ghost\ghost}
\row{\ghost\ghost}
\row{\ghost\ghost}
\row{\ghost\ghost}
\end{tableau}
\end{split}
\]
\end{example}

\begin{proof}
We can assume that $\mu\neq\lambda$, since otherwise the claim is trivial. After applying Corollary~\ref{cor:rowcol} as many times as necessary both to $\lambda,\mu$ and to $\lambda^c,\mu^c$, we can assume that $\lambda_1=(\lambda^c)_1=m$, $\mu_1<m$, $(\mu^c)_1<m$. Set $r_1,\cdots,r_m$ to be the column-lengths of $\lambda$ in decreasing order, and so define $\bv,\bw$. If we define $\bv^c,\bw^c$ similarly, then $v_i^c=v_{m-i}$ and $w_i^c=w_{m-i}$ for all $i$. We thus have an isomorphism $\Lambda_{\bv,\bw}\simto\Lambda_{\bv^c,\bw^c}$ that interchanges the roles of the maps $\Gamma_i$ and $\Gamma_{m-i}$, the maps $\Delta_i$ and $\Delta_{m-i}$, and the maps $A_i$ and $B_{m-1-i}$ for all $i$, and changes the sign of all the maps $B_i$ in order to preserve the defining equations. This induces an isomorphism $\Lambda_{\bv,\bw}^1/G_\bv\simto\Lambda_{\bv^c,\bw^c}^1/G_{\bv^c}$. So Theorem~\ref{thm:maffei} implies the desired isomorphism $\cS_\mu\cap\overline{\cO_{\lambda}}\cong\cS_{\mu^c}\cap\overline{\cO_{\lambda^c}}$.
\end{proof}

As mentioned in the introduction, Corollary~\ref{cor:rotate} provides (via Theorem~\ref{thm:juteau-mautner} again) a geometric proof of another known result in representation theory:

\begin{corollary}[Fang--Henke--Koenig~{\cite[Corollary 7.1]{fanghenkekoenig}}]
Let $p$ be a prime. With notation as in Corollary~\ref{cor:rotate}, we have $d_{\lambda\mu}^p=d_{\lambda^c\mu^c}^p$.
\end{corollary}

%%%%%%%%%%%%%%%%%%%%%%%%%%%%%%%%%%%%%%%%%%%%%%%%%%%%%%%%%%%%%%%%%%%%%%%%%%%%%%%%%%%%%%%%%%%
\section*{Acknowledgements}
I am very grateful to the organizers of the 2013 Japanese--Australian Workshop on Real and Complex Singularities (L.~Paunescu, A.~Harris, A.~Isaev, S.~Koike and T.~Fukui) for the opportunity to give the lecture series that gave rise to this article. My view of the subject has been strongly influenced by my collaborators, especially P.~Achar, D.~Juteau and E.~Sommers, the last of whom gave me valuable feedback on a draft of this article. Thanks are also due to my colleague A.~Mathas for many enlightening conversations about decomposition numbers, and to N.~Saunders for his helpful comments.
%%%%%%%%%%%%%%%%%%%%%%%%%%%%%%%%%%%%%%%%%%%%%%%%%%%%%%%%%%%%%%%%%%%%%%%%%%%%%%%%%%%%%%%%%%%%%%%%%%%%%%%

\end{document}